\documentclass[11pt,oneside]{amsart}

 \usepackage{amssymb}
\usepackage{geometry}
\usepackage{mathabx}
			\usepackage{amsmath,amscd}
			\usepackage{hyperref}
            \usepackage{fontenc}
            \usepackage{textcomp}
            \usepackage{graphicx}
\usepackage{inputenc}
      \newtheorem{theorem}{Theorem}[section]
      \newtheorem{lemma}[theorem]{Lemma}
       
      \newtheorem*{conj}{Conjecture}
      \newtheorem{coro}[theorem]{Corollary}
      
			\newtheorem*{theo}{Theorem}
			\newtheorem*{proposition}{Proposition}
			
      \newtheorem{prop}{Proposition}[section]
      \newtheorem{defi}[theorem]{Definition}

       \numberwithin{equation}{section}

      \newtheorem{remark}[theorem]{Remark}

			\newcommand{\bigO}[1]{\ensuremath{\mathop{}\mathopen{}O\mathopen{}\left(#1\right)}}

      \makeatletter
      \def\@setcopyright{}
      \def\serieslogo@{}
      \makeatother
   
\begin{document}

%



   \author{Adrien Boyer}
   \thanks{Universit\'e Paris Diderot (Paris 7)\\adrien.boyer@imj-prg.fr}


  
   \title{Spherical functions and rapid decay for hyperbolic groups
  }


   \begin{abstract}
   We investigate properties of some spherical fonctions defined on hyperbolic groups using boundary representations on the Gromov boundary endowed with the Patterson-Sullivan measure class. We prove sharp decay estimates for spherical functions as well as spectral inequalities associated with boundary representations. This point of view on the boundary allows us to view the so-called  \emph{property RD} (also called \emph{Haagerup's inequality}) as a particular case of a more general behavior of spherical functions on hyperbolic groups. In particular, we give a constructive proof using a boundary unitary representation of a result due to de la Harpe and Jolissaint asserting that hyperbolic groups satisfy  \emph{property RD}. Finally, we prove that the family of boundary representations studied in this paper, which can be regarded as a one parameter deformation of the boundary unitary representation, are slow growth representations acting on a Hilbert space admitting a proper $1$-cocycle. 
    \end{abstract}



  
   \dedicatory{}

   \date{\today}


   \maketitle



\section{Introduction and statement of the main results}

\subsection{Around property RD}
Let $\Gamma$ be discrete countable group. One of the goal of this paper is to study the operator norm of averaging operators $$\pi(f)=\sum_{\gamma \in \Gamma} f(\gamma) \pi(\gamma)$$ where $f:\Gamma \rightarrow \mathbb{C}$ is a complex-valued function and where $\pi:\Gamma \rightarrow \mathbb{B}(\mathcal{H})$ is, a priori, a non-unitary representation of $\Gamma$, that is a group morphism from $\Gamma$ to the group of invertible elements of the $C^*$-algebra of bounded operators on the Hilbert space $\mathcal{H}$ denoted by $\mathbb{B}(\mathcal{H})$. When $\pi$ is a morphism from $\Gamma$ to $\mathcal{U}(\mathcal{H})$ the group of unitary operators on $\mathcal{H}$, the representation $\pi$ is a unitary representation. 
\\
We study spectral inequalities for a Banach subspace $(E,\|\cdot\|_{E})$ of $\ell^{2}(\Gamma)$ given by
\begin{equation}
\|\lambda_{\Gamma}(f)\|_{op}=\sup_{\| v\|_{2}=1}\|f\ast v\|_{2}
\leq C \|f\|_{E},
\end{equation}
where  $\lambda_{\Gamma}$ is the left regular representation of $\Gamma$, $ \|\cdot\|_{op}$ the operator norm, $\|\cdot\|_{2}$ the $\ell^{2}$-norm on $\Gamma$, with $\ast$ denoting the convolution product and where $C$ is a real positive constant independent of $f$ an element of $E$.  More generally, one can consider a representation (a priori non-unitary) $\pi:\Gamma \rightarrow \mathbb{B}(\mathcal{H})$ and seek for the following property  

\begin{align}\label{convol}
\exists E\subset \ell^{2}(\Gamma), \exists C>0,\forall f\in E, \|\pi(f)\|_{op}\leq C \|f\|_{E}.
\end{align}

There are several natural and interesting Banach subspaces $E\subset \ell^{2}(\Gamma)$ for which one can prove inequalities of type (\ref{convol}). For instance, when $\pi=\lambda_{\Gamma}$, the space $\ell^{1}(\Gamma)$ is a such space. \\
Assume that $\Gamma$ is endowed with a length function   i.e. a map $|\cdot|:\Gamma \rightarrow \mathbb{R}^{+}$ satisfying $|e|=0$, for all $\gamma\in \Gamma, |\gamma^{-1}|= |\gamma|$ and for all $\gamma_{1},\gamma_{2}\in \Gamma, |\gamma_{1}\gamma_{2}|\leq |\gamma_{1}|+|\gamma_{2}|$. Define the Sobolev spaces for positive real numbers $d>0$ as
\begin{equation}\label{Sobo}
H^{d}(\Gamma):=\{ f:\Gamma \rightarrow \mathbb{C}| \|f\|^{2}_{H^{d}}:=\sum_{\gamma\in \Gamma} |f(\gamma)|^{2}(1+|\gamma|)^{2d}<+\infty \}.
\end{equation}
We say that $(\Gamma,|\cdot|)$ satisfies \emph{property RD} if there exists $d>0$ such that $E=H^{d}(\Gamma)$ and $\pi=\lambda_{\Gamma}$ in (\ref{convol}) for some $C>0$.\\
Let $R>0$ be a real positive number. Define a sphere of $\Gamma$ of radius $n$ of thickness $R$ as:
\begin{equation}\label{spheres}
S^{\Gamma}_{n,R}:=\{ \gamma \in\Gamma |     nR\leq |\gamma|< (n+1)R\}.
\end{equation}
 
Equivalently,  $(\Gamma,|\cdot|)$ satisfies \emph{property RD} if there exist $R>0$ and a polynomial $P$ such that for all non-negative integers $n$ and for all complex-valued functions finitely supported on $S^{\Gamma}_{n,R}$ denoted by $f$ we have 
\begin{equation}
 \|\lambda_{\Gamma}(f)\|_{op}\leq P(n)\|f\|_{2}.
 \end{equation} \\
 More generally, we say that a representation $\pi:\Gamma \rightarrow \mathbb{B}(\mathcal{H})$ satisfies \emph{RD inequality} if there exists $d>0$ such that $E=H^{d}(\Gamma)$ and $\pi=\lambda_{\Gamma}$ in (\ref{convol}) for some $C>0$. Equivalently, $\pi:\Gamma \rightarrow \mathbb{B}(\mathcal{H})$ satisfies \emph{RD inequality} if there exist $R>0$ and a polynomial $P$ such that for all non-negative integers $n$ and for all complex-valued functions $f$ finitely supported on $S^{\Gamma}_{n,R}$ we have 
\begin{equation}
 \|\pi(f)\|_{op}\leq P(n)\|f\|_{2}.
 \end{equation}\\


The property of rapid decay (property RD) was introduced by Haagerup at the end of the
seventies in his work \cite{Haa}. He proved notably that the non-abelian free groups satisfy property RD. But its essence could probably be traced back to Harish-Chandra's
estimates of spherical functions on semisimple Lie groups and to the
work of C. Herz \cite{He}. The terminology ``property RD'' was introduced later in 
the work \cite{Jo} of Jolissaint. He proved in this paper that cocompact lattices in rank one real semisimple Lie groups satisfy property RD. Afterward, de la Harpe managed to prove that hyperbolic groups satisfy this property as well \cite{dlH}. 
The first example of higher rank discrete groups having property RD is due to Ramagge, Robertson and Steger \cite{RRS}. Then, Lafforgue, inspired by the methods of \cite{RRS}, proved that cocompact lattices in $SL_{3}(\mathbb{R})$ and $SL_{3}(\mathbb{C})$ satisfy property RD. For examples of other groups satisfying property RD we refer to \cite{BP1},  \cite{BP2},  \cite{Ch2}, \cite{Ch3}, \cite{CPS}, \cite{DS}.  For more details on Property RD we refer to \cite{Chat} and \cite{Ga2}.
 \\

Indeed, the major open problem concerning property RD is Valette's conjecture:
 
\begin{conj} (The Valette conjecture)\\
Property RD holds for any discrete group acting isometrically, properly and cocompactly either on a Riemannian symmetric space or on an affine building.
\end{conj}

   Property RD is relevant in the context of the Baum-Connes conjecture.
Indeed, thanks to the important work of V. Lafforgue in \cite{La}, the Valette conjecture implies the Baum-Connes conjecture.  

One of the main point of this article is to give a new proof of a result of Jolissaint and de la Harpe saying that hyperbolic groups satisfy property RD. This is done by using boundary unitary representations. The study of the action of groups on the geometric boundary is at the heart of the proof. We use techniques coming from \emph{ergodic geometry}, inspired essentially by the papers \cite{BM} and \cite{Ro} to do \emph{coarse harmonic analysis} on boundaries of certain hyperbolic spaces. \\
Thus, we make connections with \emph{spherical functions} on hyperbolic groups and other spectral estimates. The results deal with the decay of matrix coefficients of representations (a priori non-unitary) appearing naturally on the boundary of hyperbolic groups, generalizing property RD.

\subsection{From semisimple Lie groups to hyperbolic groups}
Let $G$ be a connected semisimple Lie group with finite center and $K$ a maximal compact subgroup of $G$.  A complex-valued function $\phi$ on $G$ is called \emph{spherical} if $\phi$ is continuous, bi-$K$-invariant and satisfies the following integral condition $\int_{K}\phi(xky)dk=\phi(x)\phi(y)$ for all $x,y\in G$.  In the context of harmonic analysis of a semisimple Lie group $G$, the spherical functions are fundamental objects allowing to understand the unitary tempered dual . As a milestone, the Plancherel formula makes a very important use of them. We refer to \cite{Go},\cite{HC} and to \cite{GV} for more details on the theory of spherical functions on semisimple Lie groups.\\
We specialize now the discussion to $G=SL(2,\mathbb{R})$ and give examples of spherical functions: let $K$ be a maximal compact subgroup of $G$ and let $P$ be a minimal parabolic subgroup of $G$. The compact space $G/ P$, called the \emph{Poisson-Furstenberg boundary} (\cite{Fu}), carries a natural class of quasi-invariant finite measures under the action of $G$ on $G/ P$. We pick $\nu$ in this class, the unique $K$-invariant probability. The action $G\curvearrowright (G/ P,\nu)$ yields \emph{boundary representations} or \emph{quasi-regular representations} of $G$ denoted by $\pi_{z}(g)\in \mathbb{B}(\mathcal{H})$, where  $\mathbb{B}(\mathcal{H})$ stands for the space of bounded operators acting on the Hilbert space $\mathcal{H}=L^{2}(G/ P,\nu)$ and $z$ denotes a complex number. The representations are parametrized by $z\in \mathbb{C}$ as follows:\begin{equation}\label{represen}
\pi_{z}(g)v(\xi)=\bigg(\frac{dg_{*}\nu}{d\nu}\bigg)^{z}(\xi)v(g^{-1}\xi),
\end{equation}
with $g\in G, v\in L^{2}(G/ P,\nu)$ and $\xi \in G/ P$. When $\Re(z)=\frac{1}{2}$, the representation $\pi_{z}$ is a \emph{unitary} representation. When $s=\frac{1}{2},$ the representation $\pi_{\frac{1}{2}}$ appearing in the context of ``boundaries'' of certain space it is  also called \emph{boundary representation} and in more general contexts it is called \emph{Koopman representation} or \emph{quasi-regular representation}. This class of unitary representations has been intensively studied as such in several papers: \cite{BM}, \cite{BM2}, \cite{Boy}, \cite{BoyMa},
\cite{BPino}, \cite{Ga}, \cite{BGa}, \cite{Fink}, \cite{KS} and \cite{KS2} for boundary representations and \cite{Du1},\cite{Du2} for other quasi-regular representations.

In this paper, the family of one parameter representation defined in (\ref{repre}) might be thought of a non-unitary one parameter deformations of quasi-regular unitary representations.
Typical examples of spherical functions associated to these representations are given explicitly by 
\begin{equation}\label{sphericalf}
\phi_{z}:g\in G\mapsto \langle \pi_{z}(g)\textbf{1}_{ G/ P},\textbf{1}_{ G/ P} \rangle,
\end{equation}
where $\textbf{1}_{ G/ P}$ denotes the function equals to $1$ on the compact space $ G/ P$.
 Associated to $G$, the symmetric space $G/K$ can be identified with the hyperbolic half plane $(\mathbb{H},d_{\mathbb{H}})$, itself identified via the Cayley transform to the hyperbolic unit disc $(\mathbb{D}, d_{\mathbb{D}})$ endowed with the standard hyperbolic metric. The Poisson-Furstenberg boundary $(G/P,\nu)$ is nothing but the unit circle $\partial \mathbb{D} $ with the Lebesgue measure class though. Indeed, we can restrict the previous constructions of boundary representations and spherical functions to any lattice $\Gamma$ of $G$. The boundary representations and the spherical functions defined in (\ref{represen}) and (\ref{sphericalf}) on $G$ restrict to $\Gamma$. It turns out that property RD associated with the hyperbolic metric $d_{\mathbb{H}}$ fails for non-uniform lattices. And more generally in higher rank, property RD associated with a Riemannian metric for non-uniform lattices fails as well. Hence, the first interested case to study is a discrete group of isometries of hyperbolic spaces $\mathbb{H}$ acting cocompactly on it. This is the prototype of  a \emph{Gromov-Hyperbolic group}, or a \emph{hyperbolic group} for short. Therefore, we extend the study of some spherical functions to hyperbolic groups.\\ 
 
 Let $\Gamma$ be a hyperbolic group. This latter acts by isometries, properly discontinuously and cocompactly on a proper, geodesic and $\delta$-hyperbolic space. Recently, Nica and \v{S}pakula propose the notion of \emph{strong hyperbolicity} of a metric space:  a metric notion as a way of obtaining hyperbolicity with sharp additional properties, see Subsection \ref{sh}. It turns out that the non-elementary  hyperbolic groups act by isometries, properly discontinuously and cocompactly on such a space $(X,d)$. 

 Recall the definition of the volume growth of the group denoted by $\alpha $  defined as 
\begin{equation}
\limsup_{R \to +\infty}\frac{1}{R}\log |\Gamma \cdot o\cap B_{X}(o,R)|=\alpha .
\end{equation}

 
The geometric boundary $(\partial X,\nu_{o})$ endowed with the Patterson-Sullivan measure $\nu_{o}$ of conformal dimension $\alpha$ (associated to some basepoint $o\in X$)  where $\nu_{o}$ is quasi-invariant under the action of $\Gamma \curvearrowright \partial X$,  plays the role of the Poisson-Frustenberg boundary $(G/ P,\nu)$ in the case of semisimple Lie groups. The expressions (\ref{represen}) and  (\ref{sphericalf}) define boundary representations and spherical functions for $z\in \mathbb{C}$ for $\Gamma$ as 
$\phi_{z} :\gamma \in \Gamma \mapsto   \langle \pi_{z}(\gamma)\textbf{1}_{ \partial X},\textbf{1}_{ \partial X} \rangle.$ In this paper, we deal with only $z=s$ a real number. When the parameter $z$ has an imaginary part, our techniques do not work for the same purpose.\\
 Define for $\sigma \in \mathbb{R}$ 
  the function $\omega_{\sigma}$ as follows: for $t\in [0,+\infty[$

\begin{align}\label{defomeg}
\omega_{\sigma}(t) = \left\{
    \begin{array}{ll}
    \frac{2 \sinh\big( \sigma\alpha t \big) }{1-e^{-2\sigma \alpha}}& \mbox{if } \sigma \in \mathbb{R}^{*} \\
      t & \mbox{if } \sigma=0.
    \end{array}
\right.
\end{align}
Note that $\omega_{\sigma}$ converges to $\omega_{0}$ uniformly on all compact sets of $[0,+\infty[$, as $\sigma \to 0$  .\\
This function appears naturally in the study of the decay of $\phi_{s}$ in Section \ref{sec4}.
We obtain the following estimates.

 Given $\Gamma$ a discrete group of isometries of a metric space $(X,d)$, one define a length function $|\cdot|:\Gamma\rightarrow \mathbb{R}^{+}$ associated with a base point $o\in X$ as 
\begin{equation}\label{lengfonc}
|\gamma|:=d(o,\gamma o).
\end{equation}

\begin{prop}\label{HCHestims}

Let $(X,d)$ be a strongly hyperbolic space. Let $\Gamma$ be a discrete group of isometries of $(X,d)$ acting \emph{cocompactly} on $(X,d)$, endowed with $|\cdot|$ as in (\ref{lengfonc}).

 There exists $C>0$ such that for all $0\leq s\leq 1$, for all $\gamma \in \Gamma$
 $$C^{-1} \bigg(\omega_{|s-\frac{1}{2}|}(|\gamma|)+1\bigg) \exp\bigg({-\frac{1}{2}\alpha |\gamma|}\bigg)\leq \phi_{s}(\gamma)\leq   C \bigg(\omega_{|s-\frac{1}{2}|}(|\gamma|)+1\bigg) \exp\bigg({-\frac{1}{2}\alpha |\gamma|}\bigg),$$

\end{prop}

\begin{remark}
\begin{enumerate}
\item If $s<0$ or $s>1$, the function $\phi_{s}(\cdot)$ on $\Gamma$ is not bounded.
\item If $s=0$ or $s=1$ then $\phi_{s}(\cdot)$ equals to $\phi_{0}(\gamma)=\|\nu_{ o}\|$ and $\phi_{1}(\gamma)=\|\nu_{ o}\|$.
\end{enumerate}
\end{remark}
Note that the case $s=\frac{1}{2}$ is already well-known by results in \cite{BM}, \cite{Ga} or in \cite{Ni1} and $\phi_{\frac{1}{2}}$ is the so-called Harish-Chandra's function, rather denoted by $\Xi$.

 \subsection{Spectral transfer for amenable actions}
 To our knowledge, the use of spectral transfer can be traced back to Nevo in \cite{Ne1}, \cite{Ne2}.\\
 Given a discrete group $\Gamma$ acting on a measure space $(B,\nu)$ by quasi-preserving transformations, one consider the quasi-regular unitary representation $\pi_{\frac{1}{2}}$ defined in (\ref{repre}). 
 
 On one hand, a lemma due to Shalom \cite{Sh} ensures that for any positive finite measure $\mu$ on $\Gamma$ we have the following inequality:
 \begin{equation}\label{ineqspec1}
 \|\lambda_{\Gamma}(\mu)\|_{op}\leq \|\pi_{\frac{1}{2}}(\mu) \|_{op}.
 \end{equation}
It is easy to check that it is sufficient to prove property RD only on  positive finitely supported functions.  
 Thus, if $\pi_{\frac{1}{2}}$ satisfies RD inequality it follows that $\Gamma$ satisfies property RD.\\
   On the other hand, if $\pi_{\frac{1}{2}}$ is weakly contained in $\lambda_{\Gamma}$, namely  
   \begin{equation}\label{weaklycontained}
   \|\pi_{\frac{1}{2}}(f)\|_{op}\leq \|\lambda_{\Gamma}(f)\|_{op}
   \end{equation}
    for all $f\in \ell^{1}(\Gamma)$, then $\Gamma$ satisfies property RD implies that $\pi_{\frac{1}{2}}$ satisfies RD inequality. 
   In other words, assuming that $\pi_{\frac{1}{2}}$ is weakly contained in $\lambda_{\Gamma}$, property RD for $\Gamma$ implies that $\pi_{\frac{1}{2}}$ satisfies RD inequality.\\
   It turns out that the notion of amenable action, discovered by Zimmer \cite{Zi} and studied by Adams \cite{A1} and Kaimanovich \cite{Ka} in our context, is the right notion ensuring, by a result of Kuhn \cite{K}, that $\pi_{\frac{1}{2}}$ is weakly contained in $\lambda_{\Gamma}$ (\ref{weaklycontained}). See also \cite{ADR} for the weak containment.
   
 Hence, we have  the well-known characterization of Property RD in terms of quasi-regular unitary representations:
 
 \begin{proposition}
 Let $\Gamma$ be a discrete group acting amenably on a measure space $(B,\nu)$ by quasi-preserving transformations. Then $\Gamma$ has property RD   if and only if the boundary unitary representation $\pi_{\frac{1}{2}}$ satisfies RD inequality.
 \end{proposition}
 Since we have defined a one parameter family of boundary representations $(\pi_{s})_{s\in \mathbb{R}}$ in (\ref{repre}), the above proposition allows us to view
 Property RD as a particular case of inequalities concerning the representation $\pi_{s}$ for $s\in [0,1]$ and not only for the unitary one $\pi_{\frac{1}{2}}$.\\ 
For $\sigma>0$, define the weighted spaces associated with the functions defined in (\ref{defomeg}): 


\begin{equation}
H^{d}_{\sigma}(\Gamma):=\{ f:\Gamma \rightarrow \mathbb{C}| \|f\|^{2}_{\sigma}:=\sum_{\gamma\in \Gamma} |f(\gamma)|^{2}\big(1+\omega_{\sigma}(|\gamma|)\big)^{2d}<+\infty \}.
\end{equation}
Given the family of representations $(\pi_{s})_{s\in I}$ with $I$ a subinterval of $\mathbb{R}$ containing $\frac{1}{2}$,
 we prove in this paper inequalities of the following type:
\begin{equation}\label{defRDdeform}
 \|\pi_{s}(f)\|_{op}\leq C \|f\|_{H^{d}_{\sigma(s)}},
 \end{equation}
 where $s$ runs over $I$ and $\sigma(s)$ is a real number depending on $s$, such that for $s=\frac{1}{2}$ the associated space $H^{d}_{\sigma(s)}$, is nothing but the Sobolev space defined in (\ref{Sobo}). Hence, one can view the above inequalities as a one parameter deformation of property RD.\\

\subsection{Results}
\subsubsection{Spectral inequalities}
Our main result asserts that the growth of the operator norm $\|\pi_{s}(f)\|_{op}$ is intimately related to the growth of spherical functions $\phi_{s}(\cdot)$.
	\begin{theorem}\label{maintheo}
	Let $(X,d)$ be a strongly hyperbolic space. Let $\Gamma$ be a discrete group of isometries of $(X,d)$ acting \emph{cocompactly} on $(X,d)$. Then there exist $R,C>0$ sufficiently large such that for any $s\in [0,1]$, for all non-negative integers $n$ and for all finitely supported  $f$ on $S^{\Gamma}_{n,R}$ we have:
	
	 $$
	 \|\pi_{s}(f)\|_{op}\leq C\bigg(1+\omega_{|s-\frac{1}{2}|}(nR )\bigg) \|f\|_{2}.$$
	 	\end{theorem}
		\begin{remark}\label{optimal}
		The spectral inequality in Theorem \ref{maintheo} is optimal.
		\end{remark}

	Specialize the above inequality to $s=\frac{1}{2}$ to obtain:
	\begin{coro}
	A hyperbolic group satisfies property RD.

	\end{coro}
	
\subsubsection{1-cohomology of slow growth representations}
	The notion of slow growth representations appear in the work of Julg \cite{Ju} and Lafforgue \cite{La3}. This notion is relevant in \emph{Baum-Connes conjecture with coefficients} and inspires Lafforgue to define  the \emph{ strong property (T)}, see \cite{La4}. \\	
Now, we recall briefly the notions of slow growth representation and 1-cocycle associated to a representation  $\pi:\Gamma\rightarrow \mathbb{B}(\mathcal{H})$.	Let $\Gamma$ be a discrete countable group endowed with a length function denoted by $|\cdot|$.  The representation $\pi:\Gamma \rightarrow\mathbb{B}(\mathcal{H})$ is  a \emph{slow growth representation} or is \emph{of $\varepsilon$-exponential type} with $\varepsilon>0$ if there is a constant $C>0$ such that for all $\gamma \in \Gamma$,

$$\|\pi(\gamma)\|_{op}\leq Ce^{\varepsilon |\gamma|}.$$

A $1$-cocycle associated with $\pi:\Gamma \rightarrow\mathbb{B}(\mathcal{H})$ is a map $b:\Gamma \rightarrow \mathcal{H}$ satisfying
$$b(\gamma_{1} \gamma_{2})=b(\gamma_{1})+\pi(\gamma_{1})b(\gamma_{2}),$$ for all $\gamma_{1},\gamma_{2}\in \Gamma$.
Moreover we say that a 1-cocycle is \emph{proper} if $\|b(\gamma)\|\to+\infty $ as $|\gamma|\to +\infty.$

	Specializing the right hand inequality of Theorem \ref{maintheo} to $f$ a unit Dirac mass centered at a point $\gamma \in \Gamma$, we obtain:
		\begin{coro} Let $(X,d)$ be a strongly hyperbolic space. Let $\Gamma$ be a discrete group of isometries of $(X,d)$ acting \emph{cocompactly}. Consider $|\cdot|$ the length function defined (\ref{lengfonc}).
	For $s\in [0,\frac{1}{2}[\cup ]\frac{1}{2},1]$, the representations $\pi_{s}$ are $\alpha |\frac{1}{2}-s|$-type representations.
	\end{coro}
	
	Besides, we obtain the following theorem, which is essentially a reformulation of a nice theorem due to Nica \cite{Ni}. 

	\begin{theorem}\label{theolast}  Let $(X,d)$ be a strongly hyperbolic space. Let $\Gamma$ be a discrete group of isometries of $(X,d)$ acting \emph{cocompactly}. Consider $|\cdot|$ the length function defined (\ref{lengfonc}). Then, there exist $C>0$ and $\frac{1}{2} >\varepsilon > 0$ such that $\Gamma$ admits a family of slow growth representation acting on a Hilbert space, $\rho_{s}:\Gamma \rightarrow \mathbb{B}(\mathcal{H})$ with $s\in [0,\frac{1}{2}]$ satisfying for all $\gamma \in \Gamma$ $$\|\rho_{s}(\gamma)\|_{op}\leq C\frac{e^{|1-2s|\alpha |\gamma|}}{(1-e^{2s-1})^{2}},$$ such that $\rho_{\frac{1}{2}}:\Gamma \rightarrow \mathcal{U}(\mathcal{H} )$ is an unitary representation and  such that $\rho_{s}$ admits a proper 1-cocycle for $s\in [0,\frac{1}{2}-\varepsilon]$. 
\end{theorem}

It is worth noting that Theorem \ref{theolast} is motivated by a conjecture due to Shalom:
	\begin{conj}(Y. Shalom)
	A hyperbolic group admits a uniformly bounded representation acting on a Hilbert space with a proper 1-cocycle.
	\end{conj}

\subsection{Organization of the paper}

The paper is organized as follows. In Section \ref{sec2} we discuss all the preliminaries, explain the necessary facts involving $\delta$-hyperbolic spaces, hyperbolic groups, strongly hyperbolic spaces with the notions of roughly geodesic space and $\epsilon$-good space, Busemann functions, Patterson-Sullivan measures and shadows. Section \ref{sec3} introduces the process of discretization of the space, of the group and of the boundary that we shall use to prove our main theorem.
Section \ref{sec4} recalls briefly the definition of quasi-regular representations, property RD, Harish-Chandra function and establishes the decay of spherical functions. 
Then, Section \ref{sec5}, Section \ref{sec6} and Section \ref{sec7} are three technical sections. In Section \ref{sec5}, we introduce a dense subset of the $L^{2}$-space of the boundary. In Section \ref{sec6}, we make use of the assumption of the cocompacity of the action of the group on the space, which is absolutely fundamental for proving property RD. In Section \ref{sec7}, we use some techniques from negative curvature to obtain counting arguments. Section \ref{sec8} provides the proof of the main theorem. And finally, Section \ref{sec9} contains a proof of Theorem \ref{theolast}.

\subsection*{Acknowledgements}
I  wish to thank to Christophe Pittet, Uri Bader and Kevin Boucher for useful discussions.
I  am also grateful to Jean-Claude Picaud, Vladimir Finkelshtein, and Bogdan Nica for their remarks and comments after a careful reading of this manuscript. I am particularly grateful to Christophe Pittet for having pointed out a gap in the proof of Proposition \ref{lastcounting} in an earlier version of this manuscript.

\section{Preliminaries}\label{sec2}

\subsection{$\delta$-hyperbolic spaces} 
A metric space \((X,d)\) is said to be \emph{Gromov hyperbolic}, or \emph{$\delta$-hyperbolic} for short, if for any \(x,y,z\in X\) and some/any\footnote{if the condition holds for some \(o\) and \(\delta\), then it holds for any \(o\) and \(2\delta\)} basepoint \(o\in X\) one has
\begin{equation}\label{hyp}
  (x,y)_{o}\geq \min\{ (x,z)_{o},(z,y)_{o}\}-\delta,
\end{equation}
where \((x,y)_{o}\) stands for the \emph{Gromov product} of \(x\) and \(y\) with respect to \(o\), that is
\begin{equation}
  (x,y)_{o}=\frac{1}{2}(d(x,o)+d(y,o)-d(x,y)).
\end{equation}
Recall that an a map  $\phi:(X,d_{X}) \mapsto (Y,d_{Y})$ between metric spaces is a \emph{quasi-isometry} if there exist positive constants $L,C>0$ so that 
\begin{equation}
\frac{1}{L}d_{X}(x,y)-C \leq d_{Y}(\phi(x),\phi(y))\leq L d_{X}(x,y)+C.
\end{equation}
If we consider to the class of \emph{geodesic metric spaces}, the notion of hyperbolicity becomes invariant under quasi-isometries, which is not the case for arbitrary metric spaces.\\

\subsubsection{Gromov boundary and Bordification} 
A sequence $(a_{n})_{n\in  \mathbb{N}}$ in $X$ converges at infinity if $(a_{i},a_{j})_{o}\rightarrow +\infty$ as $i,j$ goes to $+\infty$. We say that two sequences $(a_{n})_{n\in  \mathbb{N}}$ and $(b_{n})_{n\in  \mathbb{N}}$ are equivalent if $(a_{i},b_{j})_{o} \rightarrow +\infty$ as $i,j$ goes to $\infty$. An equivalence class of $(a_{n})_{n\in  \mathbb{N}}$ is denoted by $\lim a_{n}$ and we denote by $\partial X$ the set of equivalence classes. These definitions are independent of the choice of a basepoint $o$. It turns out that the Gromov product extends to the bordification $\overline{X}:=X\cup \partial X$ by 
\begin{equation}\label{gromovextended}
(\xi,\eta)_{o}:= \sup \lim_{i,j}(a_{i},b_{j})_{o} 
\end{equation}
where the $\sup$ is taken over all sequences $(a_{n})_{n\in \mathbb{N}},(b_{n})_{n\in \mathbb{N}}$
 such that $\xi=\lim_{i}a_{i}$ and $\eta=\lim_{j}b_{j}$.
 
\begin{prop}\label{propGromov} \cite[3.17 Remarks, p. 433]{BH}.\\
 Let $X$ be a $\delta$-hyperbolic space and fix a base point $o$ in $ X$.
 \begin{enumerate}

\item The extended Gromov product $(\cdot , \cdot)_{o}$ is continuous on $X \times X$ , but not necessarily on $ \overline{X} \times \overline{X}.$
\item In the definition of $(a,b)_{o}$, if we have $a$ in $X$ (or $b$ in $X$), then we may always take the respective sequence to be the constant value $a_i = a$ (or $b_{j} =b$).
\item For all $v, w$ in $\overline{X}$ there exist sequences $(a_n)$ and $(b_n)$ such that $v=\lim a_{n}$ and $w=\lim b_{n}$ and $(\xi,\eta)_o=\lim_{n \rightarrow +\infty} (a_n,b_n)_o$.

\item For all $\xi,\eta$ and $u$ in $\overline{X}$ by taking limits we still have $$(\xi,\eta)_o \geq \min{\{(\xi,u)_{o},(u,\eta)_{o} \} }-2\delta .$$
\item For all $\xi,\eta$ in $\partial X$ and all sequences $(a_i)$ and $(b_j)$ in X with $\xi= \lim a_ i$ and $\eta= \lim b_j$, we have:

$$(\xi,\eta)_{o} -2\delta \leq  \liminf_{i,j}(a_i ,b_j)_{o} \leq (\xi,\eta)_{o}.$$ 

\end{enumerate}
\end{prop}
We refer to \cite[8.- Remarque, Chapitre 7, p. 122]{G} for a proof of the statement (5). Assuming, that $X$ is proper, the boundary $\partial X$ can be given a topology so making it compact. Moreover, the boundary $\partial X$ carries a family of \emph{visual metrics}, depending on \(d\) and a real parameter \(\epsilon > 0\) denoted for now $d_{o,\epsilon}$. For $1<\epsilon \leq \frac{\log(2)}{4\delta} $ then \(d_{o,\epsilon}\) is a metric, relative to a base point o, on \(\partial X\) satisfying
\begin{equation}\label{eq:visual-metric-def}
(3-2e^{\delta \epsilon})e^{-\epsilon (\xi,\eta)_o}\leq   d_{o,\epsilon}(\xi,\eta)\leq  e^{-\epsilon (\xi,\eta)_o}.
\end{equation}

\subsection{Hyperbolic groups}
Let $(X,d)$ be a proper space and $\Gamma$ be a  subgroup of $\mbox{Iso}(X,d)$ the group of isometries of $(X,d)$ acting properly discontinuously on $X$ i.e. for any compacts $K$ and $L$ of $X$, the set $\{ \gamma\in \Gamma | \gamma K\cap  L \neq \varnothing \}$ is finite. It is easy to see that the group $\Gamma$ is countable. Endow $\mbox{Iso}(X,d)$ with the compact open topology and thus the assumption of having a proper discontinuous action of $\Gamma$ of $X$ is equivalent to assume that $\Gamma$ is a discrete group of $\mbox{Iso}(X,d)$. \\
  We say that $\Gamma$ acts on $X$ cocompactly if $X\backslash \Gamma$ is compact for the quotient topology.
  \\
  For example if $\Gamma$ is finitely generated, one may consider its Cayley graph associated with a finite symmetric system of generators $S$: the elements of the group are the set of vertices and the edges correspond to the pair $(\gamma,\gamma')$ such that $\gamma^{-1}\gamma'\in S$. Endow the Cayley graph with the word metric  defined as 
  $d_{S}(\gamma,\gamma'):=\{\min n| \gamma^{-1}\gamma'=s_{1}\cdots s_{n} \mbox{ with } s_{1},\cdots,s_{n}\in S \}$
   and let $\Gamma$ acts on it by isometries, properly discontinuously and cocompactly.
  
  We say that $\Gamma$ is \emph{Gromov hyperbolic} or \emph{hyperbolic} for short, if it acts by isometries, properly discontinuously and cocompactly on a proper, geodesic $\delta$-hyperbolic metric space.\\

  \subsection{Balls and spheres}
Given $\Gamma$ a discrete group of isometries of a metric space $(X,d)$, recall that one can define a length function $|\cdot|:\Gamma\rightarrow \mathbb{R}^{+}$ associated with a base point $o\in X$ as in (\ref{lengfonc}). Let $R>0$ be a real positive number. Define a ball of $X$ of radius $R$ centered at a base point $o\in X$ as $B_{X}(o,R):=\{ x \in X |    d(o,x)\leq R\}$ and \emph{a sphere of $X$ of radius $n$ (a non-negative integer), of thickness $R$, centered at a base point $o\in X$} as \begin{equation}\label{spheresX}
S_{n,R}(o):=\{ x \in X |     nR\leq d(o,x)< (n+1)R\},
\end{equation}
as well as a sphere of $\Gamma$ of radius $n$ of thickness $R$ centered at a base point $o\in X$:
\begin{equation}\label{spheres}
S^{\Gamma}_{n,R}(o):=\{ \gamma \in\Gamma |     nR\leq d(o,\gamma o)< (n+1)R\}.
\end{equation}
In the following, we will use the notation $S_{n,R}$ and $S^{\Gamma}_{n,R}$ rather than $S_{n,R}(o)$ and $S^{\Gamma}_{n,R}(o)$, after having picked a base point $o$.
 Then, one can write
\begin{equation}\label{decompo}
\Gamma=\coprod_{n\geq 0} S^{\Gamma}_{n,R}(o).
\end{equation}

  \subsection{Roughly geodesics, Good $\delta$-hyperbolic spaces, Strongly hyperbolic spaces}\label{sh}
The classical theory of $\delta$-hyperbolic spaces works under the assumption that the spaces are geodesic. In general, it turns out that the Gromov product associated with a word metric on a Cayley graph of a Gromov hyperbolic group does not extend continuously to the bordification. Nevertheless, there exist metrics on $\Gamma$, such as the Mineyev metric \cite{Min} or the Green metric \cite{BHM} so that the Gomov product extends continuously to the bordification. The price to pay is that the group endowed with this new metric cannot be regarded as a geodesic metric space but rather as a \emph{roughly geodesic} metric space. In this paper we take advantage of notions of \emph{ roughly geodesic, $\epsilon$-good hyperbolic spaces and strongly hyperbolic spaces} introduced in \cite{NS} that make the computation concerning Gromov products, Busemann functions and visual metrics easier.\\
We assume now the space $(X,d)$ is a proper space. 

\begin{defi}\label{roughiso}
A metric space $(X,d)$ is roughly geodesic  if there exists $C_X>0$ so that for all $x,y\in X$ there exists a roughly isometry i.e. a map $r:[a,b]\subset \mathbb{R}\rightarrow X$ with $r(a)=x$ and $r(b)=y$ such that $ |t-s|-C_X \leq d(r(t),r(s))\leq  |t-s|+C_X$ for all $t,s\in [a,b]$.
\end{defi}
 
We say that two roughly geodesic rays 
$r,r:[0,+\infty)\rightarrow X$ are equivalent if \\ $\sup_{t}d(r(t),r'(t))<+\infty$. We  write $\partial_{r} X$ for the set of equivalence classes of roughly geodesic rays. Since $(X,d)$ is a proper roughly geodesic space, one can identify $\partial X$ to $\partial_{r} X$.

\begin{defi}(Nica-\v{S}pakula)
 We say that a hyperbolic space $X$ is $\epsilon$-good, where $\epsilon>0$, if the following
two properties hold for each base point $o\in X$:
\begin{itemize}
\item The Gromov product $(\cdot,\cdot)_{o}$ on $X$ extends continuously to the bordification $X\cup \partial X$.

\item  The map $(\xi,\eta)\in \partial X \mapsto \exp(-\epsilon(\xi,\eta)_{o}) $ is an actual metric on the boundary $\partial X$.
\end{itemize}
\end{defi}

The topology on the boundary induced by the visual metric of an $\epsilon$-good space 
is the same as the natural topology introduced in . The space metric $(\partial X,d_{o,\epsilon})$ is then a compact space. A ball on the boundary, centered at $\xi$ of radius $r$ with respect to $d_{o}$ is denoted by $B(\xi,r):=\{ \eta\in \partial X | d_{o}(\xi,\eta)<r\}$.\\

Moreover the bordification $\overline{X}:=\partial X \cup X$ is then a compactification of the space $X$.\\
Indeed, in \cite{NS} the authors introduce the notion of strong hyperbolicity. 
\begin{defi}
A metric space $(X,d)$ is \emph{$\epsilon$-strongly hyperbolic} if for all $x,y,z,o$ we have 
$$
\exp(-\epsilon (x,y)_{o})\leq \exp(-\epsilon (x,z)_{o})+\exp(-\epsilon (z,y)_{o}).
$$
\end{defi}
Then, the authors prove the following
\begin{theo}(Nica-\v{S}pakula)
An $\epsilon$-strongly hyperbolic space is a 
 $\epsilon$-good, $\frac{\log(2)}{\epsilon}$-hyperbolic space.
\end{theo}

An example of such spaces is the class of CAT(-1) spaces which are $1$-good geodesic metric spaces. In the context of CAT(-1) spaces, the formula 
\begin{equation}\label{distance}
	d_o(\xi,\eta)={\rm e}^{-(\xi,\eta)_o}
\end{equation}
(we set $d_o(\xi,\xi)=0$). This is due to M. Bourdon, we refer to \cite[ 2.5.1 Th\'eor\`eme]{Bou} for more details. Hence the CAT(-1) spaces are examples of $1$-good hyperbolic spaces.\\

The main point is the following theorem: that is a combination of results due to Blach\`ere, Ha\" issinky and Matthieu \cite{BHM} and of Nica and \v{S}pakula \cite{NS}. 

\begin{theo}
A hyperbolic group acts by isometries, properly discontinuously and cocompactly on a roughly geodesic  $\epsilon$-good $\delta$-hyperbolic space.
\end{theo}

A concrete example of such space is the group itself endowed with the Mineyev metric \cite{Min} or the Green metric associated to a random walk. Let us describe briefly the case of the Green metric: Let $\Gamma$ be a hyperbolic group. A probability measure $\mu$ on $\Gamma$ defines a random walk on $\Gamma$ with transition probability $p(\gamma,\lambda):=\mu(\lambda^{-1}\gamma)$. We say that $\mu$ is symmetric if $\mu(\gamma)=\mu(\gamma^{-1})$ and finitely supported if the support of $\mu$ is a finite generating set of $\Gamma$. Define the Green function $$G(\gamma,\lambda):=\sum_{n\geq 0}\mu^{n}(\gamma^{-1}\lambda)$$ where $\mu^{n}$ defines the $n^{th}$ convolution of $\mu$. Let $F(\gamma,\lambda)$ be the probability that a random walk starting at $\gamma$ hits $\lambda$, that is $$F(\gamma,\lambda)=\frac{G(\gamma,\lambda)}{G(e,e)}.$$ One can associate with a random walk $\mu$ on $\Gamma$, a metric on $\Gamma$ called the Green metric, defined as 
\begin{equation}\label{Greenmetric}
d_{G}(\gamma,\lambda)=-\log F(\gamma,\lambda).
\end{equation} 
The metric space $(\Gamma,d_{G})$, when $\mu$ is symmetric and finitely supported is a typical example of a proper roughly geodesic $\epsilon$-good $\delta$-hyperbolic space on which $\Gamma$ acts by isometries, properly discontinuously and cocompactly.

Indeed, using the concept of \emph{quasi-ruler}, Blach\`ere, Ha\" issinky and Matthieu prove that  $(\Gamma,d_{G})$ is a roughly geodesic $\delta$-hyperbolic space \cite{BHM} whereas Nica and  \v{S}pakula in \cite{NS} prove that this space is an $\epsilon$-strongly hyperbolic space.

\subsection{Busemann functions}
We just saw that the Gromov boundary of an $\epsilon$-good space has also a geometrical definition 
and if $\xi\in \partial X$ we can pick a roughly geodesic $r$, namely a map $r:\mathbb{R}_{+} \rightarrow X$ satisfying Definition \ref{roughiso}, such that $r(+\infty)=\xi$ to define the \emph{Busemann function associated to $r$} as
 \begin{equation*}
b_{r}(x)=\lim_{t \rightarrow +\infty} d(x,r(t))-t,
\end{equation*}
which is well defined due to the triangle inequality.

We define the \emph{horoshperical distance relative to $\xi$} as:
\begin{equation*}\beta_{\xi}(x,y)=b_{r}(x)-b_{r}(y) .
\end{equation*}
Note that this definition does not depend on the choice of a representative of $\xi$.\\ 
It turns out that, that in a hyperbolic $\epsilon$-good metric space we can write:
\begin{equation}\label{buseman}
\beta_{\xi}(x,y)=2(\xi,y)_{x}-d(x,y).
\end{equation}
Moreover we have for all $\xi \in \partial X$ and for all $x,y\in X$:
\begin{equation}\label{buseman'}
(\xi,y)_{x}\leq d(x,y),
\end{equation}
and thus 
\begin{equation}\label{buseman''}
\beta_{\xi}(x,y)\leq d(x,y).
\end{equation}

The conformal metrics $(d_{x,\epsilon})_{x\in X}$ associated with $\epsilon>0$, satisfy the following relation: for all $x,y\in X$ and for all $\xi,\eta \in \partial X$ we have:
\begin{equation}
d_{y,\epsilon}(\xi,\eta)=\exp \big(\frac{\epsilon}{2} (\beta_{\xi}(x,y)+\beta_{\eta}(x,y))\big)d_{x,\epsilon}(\xi,\eta).
\end{equation}

\subsection{The Patterson-Sullivan measure}
\label{sec:patt-sull-meas}
Let $\Gamma$ be a nonelementary discrete group of isometries of a proper $\epsilon$-strongly hyperbolic metric space $(X,d)$ and let $\overline{X}=X \cup \partial X$ be the compactification of $X$. 
Recall the definition the volume growth of the group denoted by $\alpha $  defined as 
\begin{equation}
\limsup_{R \to +\infty}\frac{1}{R}\log |\Gamma.o\cap B(o,R)|=\alpha .
\end{equation}
Observe that is does not depend on the choice of $o\in X$. It turns out that the volume growth is controlled as
\begin{equation}\label{volumegrowth}
C^{-1} e^{\alpha  R}\leq |\Gamma.o\cap B(o,R)|\leq C e^{\alpha  R},
\end{equation}
for some constant $C$ independant of $R$.\\ 
The limit set of $\Gamma$ denoted by $\Lambda_{\Gamma}$ is the set of all accumulation points in $\partial X$ of an orbit. Namely $\Lambda_{\Gamma}:=\overline{\Gamma o}\cap \partial X$, with the closure in $\overline{X}$. Notice that the limit set does not depend on the choice of a base point $o\in X$. If $\Gamma$ acts cocompactly on $X$ then $\Lambda_{\Gamma}=\partial X$. 
Consider now a visual metric  on the boundary of parameter $\epsilon$ associated with a base point $o$. 
The space $(\partial X,d_{\epsilon})$ is a compact metric space, and therefore it admits a Hausdorff measure of dimension \(D\). It is nonzero, finite, and finitely-dimensional, and in the context of hyperbolic geometry it is known as the \emph{Patterson-Sullivan} measure. We will denote it by \(\nu_{o}\), and \textbf{normalize it so that \(\nu_{o}(\partial X) = 1\)}. Usually all measures of the form \(\rho d\nu_{o}\) with \(0<c\leq \rho(\xi)\leq C<\infty\) for some \(c,C\) are called Patterson-Sullivan measures. This class of measures is independent of the choice of \(\epsilon\), but different metrics \(d\) usually give rise to mutually singular measures. Moreover, the Hausdorff dimension changes with \(\epsilon\), although the relationship is very simple, namely
\begin{equation}
  D=\frac{\alpha }{\epsilon}.
\end{equation}

The Patterson-Sullivan measure \(\nu_{o}\) is quasi-invariant under the action of  \(\Gamma\)  on $\partial X$. It actually satisfies a stronger condition in the class of $\epsilon$-good spaces, namely
\begin{equation}
  \frac{d\gamma_{*}\nu_{o}}{d\nu_{o}}(\xi)= e^{\epsilon D \beta_{\xi}(o,\gamma o)}.
\end{equation}
It turns out that the support of $\nu_{o}$ is in $\Lambda_{\Gamma}$ and moreover $\nu_{o}$ is Ahlfors regular of dimension \(D\) we have the following estimate for the volumes of balls: there exists $C_{\nu}>0$ so that for all $\xi \in \Lambda_{\Gamma}$ for all \(r \leq Diam (\partial X)\):
\begin{equation}
 C_{\nu}^{-1} r^{D}\leq  \nu_{o}(B(r,\xi)) \leq C_{\nu} r^{D}.
\end{equation}

Finally, the Patterson-Sullivan measure is ergodic for the action of \(\Gamma\), and thus unique.
The foundations of Patterson-Sullivan measures theory are in the important papers \cite{Pa} and \cite{Su}. See \cite{Bou},\cite{BMo} and \cite{Ro} for more general results in the context of CAT(-1) spaces. These measures are also called \textit{conformal densities}.
\begin{remark}
In the general context of $\delta$-hyperbolic spaces the conformal densities are \emph{quasi-conformal densities}.  A priori for general hyperbolic groups one can construct only \emph{invariant quasi-conformal densities} and this is due to Coornaert in  \cite[Th\'eor\`eme 8.3]{Co}: he proves the existence of $\Gamma$-invariant quasi-conformal densities of dimension $\alpha $ when $X$ is a proper geodesic $\delta$-hyperbolic space. Note that his construction has been extended to the case of roughly geodesic metric spaces in \cite{BHM}.
\end{remark}

\subsection{Shadows}
\subsubsection{Upper Gromov bounded by above}
This assumption appears in the work of Connell and Muchnik in \cite{CM} as well as in the work of Garncarek on boundary unitary representations \cite{Ga}. We say that a space $X$ is \emph{upper gromov bounded by above} with respect to $o$, if there exists a constant $M>0$ such that for all $x\in X$ we have
 \begin{equation} \sup_{\xi \in \partial X}(\xi,x)_{o}\geq d(o,x)-M.
 \end{equation}
 
  Morally, this definition allows us to choose a point in the boundary playing the role of the forward endpoint of a geodesic starting at $o$ passing through $x$ in the context of simply connected Riemannian manifold of negative curvature. \\
  We denote by $\theta^{x}_{o}$ a point in the boundary satisfying 
 
 \begin{equation} 
 (\theta^{x}_{o},x)_{o}\geq d(o,x)-M.
 \end{equation}
  In particular, the CAT(-1) spaces are upper Gromov bounded by above as well as hyperbolic groups endowed with a left invariant word metric associated with some finite symmetric set of generators (see \cite[Lemma 4.1]{Ga}).  

\subsubsection{Definition of shadows}
Let $(X,d)$ be a roughly geodesic, $\epsilon$-good, $\delta$-hyperbolic space. 
Let $r>0$ and a base point $o \in X$.
Define a shadow for any $x\in X$ denoted by $O_{r}(o,x)$ as 
\begin{equation}
O_{r}(o,x):=\{ \xi\in \partial X | (\xi,x)_{o}\geq d(x,o)-r\}.
\end{equation}

\begin{lemma}\label{ombre}
Assume $r>M+\delta$. Then
$$B(\theta^{x}_{o},e^{-\epsilon(d(o,x)-r+\delta)})\subset O_{r}(o,x):=\{ \xi\in \partial X | (\xi,x)_{o}\geq d(x,o)-r\} \subset B(\theta^{x}_{o},e^{-\epsilon(d(x,o)-r-\delta)}). $$
\end{lemma}
\begin{proof}
We have for $r>M+\delta$:
\begin{align*}
(\xi,x)_{o}&\geq \min \{ (\xi,\theta^{x}_{o})_{o},(\theta^{x}_{o},x)\}-\delta\\
&=\min \{d(o,x)-r+\delta,d(o,x)-M \}-\delta \\
&=d(o,x)-r.
\end{align*}
On the other hand 
\begin{align*}
(\xi,\theta^{x}_{o})_{o}&\geq \min \{ (\xi,x)_{o},(\theta^{x}_{o},x)_{o}\}-\delta\\
&\min \{ d(x,o)-r,d(o,x)-M \}-\delta\\
&\geq d(o,x)-r -\delta.
\end{align*}
\end{proof}
\textbf{For now on we fix $r>M +\delta$.}

\section{Discretization}\label{sec3}

Let $(X,d)$ be a roughly geodesic $\epsilon$-good $\delta$-hyperbolic space and $o\in X$ an origin.

\subsection{Discretization of roughly geodesics}
\subsubsection{A construction}
Let $R$ be positive real number such that  $R>5C_{X}$ where $C_{X}$ is the constant of roughly geodesics in Definition \ref{roughiso}.\\
Now if $\xi$ denotes any point in the boundary we consider the roughly geodesic $[o\xi)$, starting at $o$ ending at $\xi$ represented by a roughly-isometry $r_{o}:\mathbb{R}^{+}\rightarrow X$. Define then for $k\in \{0,\dots,n\}$:
\begin{equation}\label{zk'}
\xi_{k}:=r_{o}(kR+\tau) \in [o,\xi ),
\end{equation}
with $\tau=2C_{X}$.
By Definition \ref{roughiso}, one can write 
\begin{equation}\label{ineqzk}
kR+C_{X} \leq d(o,\xi_{k}) \leq kR +3C_{X},
\end{equation}
and thus the choice of $R$ and $\tau$ are justified by $\xi_{k}\in S_{k,R}.$ 

Now let $x$ be a point in $X$ and  $n$ be the unique non-negative integer so that $$n R\leq d(o,x)<(n+1)R.$$ Consider $\theta^{x}_{o}\in \partial X$. Let $[o,\theta^{x}_{o} )$ be a roughly geodesic starting at $o$ ending at $\theta^{x}_{o}$ represented by a roughly isometry $r'_{o}:\mathbb{R}^{+}\rightarrow X$. Apply the above constrution to $\xi=\theta^{x}_{o}$ and define for $k \in \{0,\dots, n\}$ the point $x_{k}$ satisfying 
\begin{equation}\label{zk}
x_{k}:=r_{o}(kR+\tau) \in [o,\theta^{x}_{o} ),
\end{equation}
 And if $\gamma $ is an element of a group $\Gamma$, we use the notation:
 \begin{equation}\label{zkg}
\gamma_{k}:=r_{o}(kR+\tau) \in [o,\theta^{\gamma o}_{o} ).
\end{equation}
 
 Hence we still have: \begin{equation}\label{distancezk}
kR+C_{X} \leq d(o,x_{k})\leq kR+3C_{X} \mbox{ and so } x_{k}\in S_{k,R}.
 \end{equation}
\subsubsection{Properties}

\begin{lemma}\label{ineqzkend}
For all $\xi \in \partial X$ and for all non-negative integers $k$ we have

$  kR+\frac{1}{2}C_{X} \leq (\xi_{k},\xi) _{o}\leq kR+\frac{7}{2}C_{X}.$
\end{lemma}
\begin{proof}
For $t>kR+\tau$ we have
  $t-kR-3C_{X} \leq d(r_{o}(t),x_{k})\leq t-kR-C_{X}.$ By definition of the Gromov product, $(r_{o}(t),x_{k})_{o}=\frac{1}{2}(d(o,r_{o}(t))+d(o,x_{k})-d(r_{o}(t),x_{k})).$
  Hence  $(r_{o}(t),x_{k})_{o}\leq     
  kR+\frac{7}{2}C_{X},$ and 
   $(r_{o}(t),x_{k})_{o}\geq 
   kR+\frac{1}{2}C_{X}.$
By continuity of the Gromov product to the bordification we obtain:
$$ kR+\frac{1}{2}C_{X} \leq (\xi_{k},\xi) _{o}\leq kR+\frac{7}{2}C_{X}.$$  
  \end{proof}
In particular, according to (\ref{zk})
  the above lemma implies for all $x\in X$, for all non-negative integers $k$ that 
\begin{equation}\label{zktheta}
(x_{k},\theta^{x}_{o})_{o}\geq kR+\frac{1}{2}C_{X}.
\end{equation}

\begin{lemma}\label{roughlygeo1}
Let $n$ be an integer and  let $x \in S_{n,R} $ with $R>\frac{1}{2}C_{X}+M$. For any $k\in \{0,\dots, n-1\}$, the point $x_{k}$ satisfies $(x_{k},x)_{o}\geq kR+\frac{1}{2}C_{X}-\delta$. Besides, $(x_{n},x)_{o}\geq nR-M-\delta$.

\end{lemma}
\begin{proof}
By Inequality (\ref{hyp}) of $\delta$-hyperbolicity :
\begin{align*}
(x_{k},x)_{o}&\geq \min\{(x_{k},\theta^{x}_{o})_{o},(\theta^{x}_{o},x)_{o}\}-\delta\\
&\geq \min\{(x_{k},\theta^{x}_{o})_{o},d(o,x)-M\}-\delta\\
&\geq  \min\{kR+\frac{1}{2}C_{X},nR-M\}-\delta,\\
\end{align*}
where the last inequality follows from Inequality (\ref{zktheta}). If $k\in \{0,\dots, n-1\}$, with condition $R>\frac{C_{X}}{2}+M$, we obtain $(x_{k},x)_{o}\geq kR+\frac{1}{2}C_{X}-\delta$, although for $k=n$ we obtain $(x_{n},x)_{o}\geq nR-M-\delta$, and the proof is done.
\end{proof}

\begin{lemma}\label{controlzkxx}
There exists $C>0$ so that for any $R>5C_{X}$, for all non-negative integers $n$, for all $x\in S_{n,R}$ and for all $k\in \{0,\dots,n \}$ we have: $(n-k)R-C\leq d(x_{k},x ) \leq (n-k)R+C.$
\end{lemma}

\begin{proof}

Using the definition of the Gromov product we have
\begin{align*}
 d(x_{n},x)&=d(x_{n},o)+d(x,o)-2(x_{n},x)_{o}.
  \end{align*}
 Thus, by Lemma \ref{roughlygeo1} asserting
\begin{align*}
(x_{n},x)_{o}&\geq nR-M-\delta,
\end{align*}
we obtain
 \begin{align*}
 d(x_{n},x) &\leq nR+3C_{X}+(n+1)R- 2(x_{n},x)_{o}\\
 &\leq nR+3C_{X}+(n+1)R- 2nR+2\delta+2M\\
 &=R+3C_{X}+2M+2\delta.
 \end{align*}
 
Set $C':=R+3C_{X}+2M+2\delta$. 

Triangle inequality implies for all $k\in\{0,\dots,n\}$: 
$$d(x_{k},x_{n})- d(x_{n},x)\leq d(x_{k},x)\leq d(x_{k},x_{n})+ d(x_{n},x).$$ we obtain, $$(n-k)R-C_{X}- C'\leq d(x_{k},x)\leq (n-k)R+C_{X}+C'.$$
Set $C=C_{X}+C'$ to conclude the proof.
\end{proof}




\subsection{Horospherical decomposition of the boundary} Let $\Gamma$ be a discrete group of isometries of $(X,d)$.\\

 Define for $n\geq 2$, for $R>0$, for any $\gamma \in S^{\Gamma}_{n,R}$ and for all $k\in \{1,\dots,n\}$ the sets denoted by $A_{k,R}(o,\gamma o)$ as:
 
 For $k=1$
 \begin{equation}
 A_{1,R}(o,\gamma o):=\{  \xi\in\partial X |    (\xi,\gamma o)_{o}< R \}.
 \end{equation}
  For $k\in \{1,\dots, n-1\}: $
\begin{equation}
A_{k,R}(o,\gamma o):=\{  \xi\in\partial X |    kR \leq(\xi,\gamma o)_{o}<(k+1)R \}.
\end{equation}
For $k=n$  
\begin{equation}
A_{n,R}(o,\gamma o):=\{  \xi\in\partial X |    nR\leq(\xi,\gamma o)_{o} \}.
\end{equation} 
 In the following, we denote $A_{k,R}(o,\gamma o)$ by $A_{k,R}(\gamma )$.\\
 
Hence, for all $\gamma \in S^{\Gamma}_{n,R}$, with $n$ large enough, the sets $(A_{k,R}(\gamma))_{k\in \{1,\dots,n\}}$ provide a partition of the boundary of $X$ given by: 
 \begin{equation}\label{partition}
\forall n\geq 2,\forall \gamma \in S^{\Gamma}_{n,R},\mbox{ }\partial X=\coprod^{n}_{k=1} A_{k,R}(\gamma).
 \end{equation}
 Let $R>0$ large enough.\\
\begin{lemma}\label{roughlygeo2}
There exists $C>0$ such that for all non-negative integers $n\geq 1$, for all $k \in \{1,\dots,n\}$ and for all $\xi \in A_{k,R}(\gamma)$ we have $(\xi,\gamma_{k})_{o}\geq  kR-C.$
\end{lemma}
\begin{proof}
Lemma \ref{roughlygeo1} provides a constant $C'>0$ such that for all $\gamma \in S^{\Gamma}_{n,R}$, for all $k\in \{1,\dots,n \}$ and for all $\xi \in A_{k,R}(\gamma)$:
$$(\gamma o,\gamma_{k})_{o}\geq kR-C'.$$
therefore,
\begin{align*}
(\xi,\gamma_{k})_{o}&\geq \min\{ (\xi,\gamma o)_{o},(\gamma o,\gamma_{k})_{o}\}-\delta\\
&\geq  \min\{ kR,kR-C'\}-\delta\\
&\geq  kR-C'-\delta.
\end{align*}
Then, set $C=C'+\delta$ to conclude the proof.
\end{proof}

Let $\nu=(\nu_{x})_{x\in X}$ be a Patterson-Sullivan density of $\Gamma$, of conformal dimension $\alpha$. Assume that $(\partial X,d_{o,\epsilon},\nu_{o})$ is $\alpha$-Ahlfors regular. We investigate the measure of the sets of the horospherical decomposition of the boundary using Ahlfors regularity of $\nu_{o}$.

\begin{lemma}\label{ahlforsAk} There exist $C>0$ and $R>0$ such that for all non-negative integers $n\geq 1$  and for all $\gamma$ in $S^{\Gamma}_{n,R}$ and for all integers $k$ in $\{1,\dots,n\}$ and for all non-negative integer $n\geq 1$: 
$$C^{-1}e^{-\alpha k R}\leq \nu_{o}(A_{k,R}(\gamma ))\leq Ce^{-\alpha k R}.$$
\end{lemma}
\begin{proof}
Let $R>M+\delta$. We shall write precisely the sets $A_{k,R}(\gamma)$ in terms of shadows with $\gamma \in S^{\Gamma}_{n,R}$ and $n\geq 1$.\\
For $k=1$ write for $|\gamma|>R$ that $A_{1,R}(\gamma)=\partial X \backslash O_{r_{1}}(o,\gamma o)$ with $r_{1}=|\gamma|-R$. Then 

$$1-c\leq \nu_{o}(A_{1,R}(\gamma))\leq 1,$$
with $1>c>0$.
\\
For $k\in \{2,\dots,n-1\}$, write $A_{k,R}(\gamma):=O_{r_{k}}(o,\gamma o) \backslash O_{r_{k+1}}(o,\gamma o)$ with $r_{k}=|\gamma|-kR$. Note that $r_{k}>M+\delta$ since we choose $R>M+\delta$. Then, Lemma \ref{ombre} implies:

$$ B(\theta_{o}^{\gamma o},e^{-\epsilon(|\gamma|-r_{k}+\delta)}) \backslash B(\theta_{o}^{\gamma o},e^{-\epsilon (|\gamma|-r_{k+1} -\delta)}) \subset A_{k,R}(\gamma) $$ and 
$$A_{k,R}(\gamma) \subset B(\theta_{o}^{\gamma o},e^{-\epsilon(|\gamma|-r_{k}-\delta)}) \backslash B(\theta_{o}^{\gamma o},e^{-\epsilon(|\gamma|-r_{k+1} +\delta)}).$$

Thus, Ahlfors regularity of $\nu_{o}$ implies:  
\begin{align*}
C_{\nu}^{-1}e^{-\alpha \delta}e^{-\alpha kR }-C_{\nu}e^{-\alpha (R-\delta)}e^{-\alpha kR }\leq \nu_{o}\big(A_{k,R}(\gamma)\big)\leq C_{\nu}e^{\alpha \delta}e^{-\alpha kR } -C_{\nu}^{-1}e^{-\alpha( R+\delta)}e^{-\alpha kR }\\
\big(C_{\nu}^{-1}e^{-\alpha \delta}-C_{\nu}e^{-\alpha (R-\delta)}\big)e^{-\alpha kR }\leq \nu_{o}\big(A_{k,R}(\gamma)\big)\leq \big(C_{\nu}e^{\alpha \delta} -C_{\nu}^{-1}e^{-\alpha( R+\delta)}\big)e^{-\alpha kR }
\end{align*}
Then, choose $R>0$ large enough such that the quantities
\begin{equation}\label{constants}
C_{\nu}^{-1}e^{-\alpha \delta}-C_{\nu}e^{-\alpha (R-\delta)},C_{\nu}e^{\alpha \delta} -C_{\nu}^{-1}e^{-\alpha( R+\delta)}>0.
\end{equation}
For $k=n$ write $A_{n,R}(\gamma)=O_{r_{n}}(o,\gamma o)$, hence
$$ C_{\nu}^{-1}e^{-\alpha \delta}e^{-\alpha nR}\leq \nu(A_{n,R}(\gamma))\leq C_{\nu}e^{\alpha \delta}e^{-\alpha nR}.$$ By comparing the constants in (\ref{constants}) with the constant in the above inequalities, we find the desired constant $C>0$ for $R>0$ large enough.

\end{proof}

\begin{lemma}\label{lesAi}
For any $R>0$, for all non-negative integers $n\geq 1$ and for all $\gamma \in S^{\Gamma}_{n,R}$, we have for all $k=1,\dots,n-1$
$$A_{n-k,R}(\gamma^{-1} ) \cup A_{n-k+1,R}(\gamma^{-1}) \subset \gamma^{-1} A_{k,R}(\gamma )\subset  A_{n-k-1,R}(\gamma^{-1} ) \cup A_{n-k,R}(\gamma^{-1}), $$ and for $k=n-1$ we have $$ \gamma^{-1}A_{n,R}(\gamma )\subset A_{1,R}(\gamma^{-1} ).$$
Moreover, $$ A_{n-k,R}(\gamma^{-1} )\subset \gamma ^{-1}A_{k-1,R}(\gamma ) \cup  \gamma^{-1} A_{k,R}(\gamma)$$
\end{lemma} 
\begin{proof}
 By the continuity of the Gromov product to the bordification, we have for all $\xi \in \partial X$
and for all $x,y\in X$: $(\xi,x)_{y}+(\xi,y)_{x}=d(x,y)$. Thus we obtain for all $\gamma \in \Gamma$ and for all $\xi\in \partial X:$
\begin{align*}
(\gamma^{-1}\xi,\gamma^{-1}o)_{o}=-(\xi,\gamma o)_{o}+|\gamma|.
\end{align*}
Let $\xi \in A_{k,R}(\gamma )$. 
Thus  $$|\gamma|-(k+1)R\leq (\gamma^{-1}\xi,\gamma^{-1}o)_{o}\leq |\gamma|-kR.$$So if $\gamma \in S^{\Gamma}_{n,R}$ we obtain for $k=1,\dots,n$ 
$$(n-k-1)R\leq (\gamma^{-1}\xi,\gamma^{-1}o)_{o}< (n-k+1)R.$$
That is to say $\gamma^{-1}A_{k,R}(\gamma )\subset  A_{n-k-1,R}(\gamma^{-1} ) \cup A_{n-k,R}(\gamma^{-1})$.

 And for $k=n$ we have $$ \gamma^{-1}A_{n,R}(\gamma )\subset A_{1,R}(\gamma^{-1} ).$$
\end{proof}

On the sets $A_{k,R}(\gamma )$ for $k=1,\dots, n$, we have sharp estimates of Busemann functions. 

\begin{lemma}\label{BusemannsurAi}
For any $R>0$, for all non-negative integers $n\geq 1$ and for all $\gamma$ in $S^{\Gamma}_{n,R}\subset \Gamma$ we have
for all $\xi\in A_{n,R}(\gamma)$:
$$ nR-R\leq \beta_{\xi}(o,\gamma o)\leq   nR+R,$$
and for all $k=\{2,\dots,n-1\}$ and for all $\xi \in A_{k,R}(\gamma )$:
$$(2k-n)R-R\leq \beta_{\xi}(o,\gamma o)\leq   (2k-n)R+R,$$ and for all $\xi \in A_{1,R}(\gamma ):$
$$-nR -R\leq \beta_{\xi}(o,\gamma o) \leq -nR +2R.$$
\end{lemma}
\begin{proof}
Write $\beta_{\xi}(o,\gamma o)=2(\xi,\gamma o)_{o}-d(o,\gamma o)$. Then by definition of 
$A_{k,R}(\gamma )$ we have $ k R\leq(\xi,\gamma o)_{o}\leq (k+1) R$ for $k=2,\dots, n-1$. Thus $$ (2k-n)R-R\leq \beta_{\xi}(o,\gamma o)\leq (2k-n)R+R.$$ 
Besides, if $\xi \in A_{n,R}(\gamma)$ we have
$$2(n-1)R-d(o,\gamma o)\leq \beta_{\xi}(o,\gamma o)\leq d(o,\gamma o).$$
Since $\gamma \in S^{\Gamma}_{n,R}$ we have $nR\leq|\gamma|<(n+1)R$
$$2(n-1)R-(n+1)R\leq \beta_{\xi}(o,\gamma o)\leq (n+1)R.$$

Finally, for  $\xi \in A_{1,R}(o,\gamma o)$ the following estimates hold $$ -(n+1)R\leq-|\gamma|\leq\beta_{\xi}(o,\gamma o)\leq -nR+2R.$$

\end{proof}

\section{Boundary representations and Spherical functions}\label{sec4}
\subsection{Quasi-regular representations and spherical functions}
Given a measure space $(B,\nu)$ with a measure class preserving action of locally compact group $\Gamma$, one consider a family of representations $\pi_{s}:\Gamma \rightarrow \mathcal{U}(L^{2}(b,\nu))$ defined by:
\begin{equation}\label{repre}
\pi_{s}(\gamma) v(b)=\bigg(\frac{d\gamma_{*}\nu}{d\nu}\bigg)^{s}v(\gamma^{-1}b),
\end{equation}
where $s$ is a  positive real number. When $s=\frac{1}{2}$, the representation $\pi_{\frac{1}{2}}$ is an unitary representation. When the latter appears in the context of ``boundaries'' of certain space or in ergodic theoretic context, it is also called quasi-regular representation (or Koopman representation) and it has been intensively studied as such in several papers: \cite{BM}, \cite{BM2}, \cite{Ga}, \cite{Boy}, \cite{BoyMa},
\cite{BPino}, \cite{BGa}, \cite{Fink}, \cite{KS} and \cite{KS2} for boundary representations and see \cite{Du1},\cite{Du2} for other quasi-regular representations.

In this paper, the family of a one parameter representation defined in (\ref{repre}), can be thought as an one parameter non-unitary deformation of the unitary quasi-regular representation.


\subsection{Spherical functions on hyperbolic groups} \label{sphericalfunctions}Given a non-elementary discrete groups $\Gamma$ acting by isometries of $(X,d)$ a roughly geodesic $\epsilon$-good $\delta$-hyperbolic space, we apply the above construction to $\Gamma \curvearrowright (\partial X,\nu_{o})$ where $\nu_{o}$ is the Patterson-Sullivan measure associated with a base point $o$.\\
The adjoint representation $\pi^{*}_{s}(\gamma)$ satisfying $\langle \pi_{s}(\gamma)v,w\rangle = \langle v,\pi^{*}_{s}(\gamma)w\rangle$ is given by:
\begin{equation}
\pi^{*}_{s}(\gamma)=\pi_{1-s}(\gamma^{-1}).
\end{equation}
This is of interest of studying the spherical function in the context of harmonic analysis Lie groups (we refer to \cite{GV} for more details on spherical functions), hence we define a spherical function associated with $\pi_{s}$ as the matrix coefficient: \begin{align}
\phi_{s}:\gamma \in \Gamma \mapsto \langle \pi_{s}(\gamma)\textbf{1}_{\partial X}, \textbf{1}_{\partial X}\rangle \in \mathbb{R}^{+}.
\end{align}
To our knowledge, few is known in this generality and in this context of general hyperbolic groups except for $s=\frac{1}{2}$. In the latter case, $\phi_{\frac{1}{2}}$ is analogous to  Harish-Chandra's $\Xi$ function in the context of  Lie groups, rather denoted  by $\Xi$ instead of $\phi_{\frac{1}{2}}$. One can prove the following estimates, called \emph{ Harish-Chandra Anker estimates}, naming related to  \cite{Ank}: there exists $C>0$ such that for all $\gamma \in \Gamma:$
\begin{align}\label{HCHestimates}
C^{-1}(1+|\gamma|)\exp{\bigg(-\frac{1 }{2}\alpha |\gamma|\bigg)} \leq \Xi(\gamma)\leq C (1+|\gamma|)\exp\bigg({-\frac{1 }{2}\alpha |\gamma|}\bigg).
\end{align}

Recall the definition (\ref{defomeg}) of the function  $\omega_{\frac{1}{2}-s}(\cdot)$ for $s\in \mathbb{R} \backslash \{\frac{1}{2}\}$, 

$$\omega_{\frac{1}{2}-s}(t) =\frac{2 \sinh\big( (\frac{1}{2}-s)\alpha t \big) }{1-e^{(2s-1)\alpha}}$$
and observe that $\omega_{\frac{1}{2}-s}(\cdot)$ converges uniformly on compact sets of $\mathbb{R}$ to 
    $ \omega_{0}(t)=t $ as $s\to \frac{1}{2}$.  \\
     
    It turns out that estimates (\ref{HCHestimates}) are limits of more general estimates established in the following proposition. \\




\begin{remark}
It is worth noting that Proposition  \ref{HCHestims} implies in particular that $\phi_{s}(\gamma)$ satisfies, if $s<\frac{1}{2}$ \begin{equation} \phi_{s}(\gamma)= \bigO{\frac{e^{-s\alpha |\gamma|}}{1-e^{(2s-1)\alpha}}}
\end{equation}
 when $|\gamma|\to +\infty$ and  if $s>\frac{1}{2}$ \begin{equation}\phi_{s}(\gamma)=\bigO{\frac{e^{(s-1)\alpha |\gamma|}}{e^{(2s-1)\alpha}-1}}.
\end{equation}

\end{remark}

\begin{proof}
We assume that $s\neq \frac{1}{2}$, the case of equality is known. By definition, $\phi_{s}(\gamma)=\int_{\partial X}e^{s\alpha \beta_{\xi}(o,\gamma o)}d\nu_{o}(\xi).$ Using the relation (\ref{buseman}) we have
\begin{align*}
\phi_{s}(\gamma)&=e^{-s\alpha |\gamma|}\int_{\partial X}e^{2s\alpha (\xi,\gamma o)_{o}}d\nu_{o}(\xi).
\end{align*}
Decompose $\Gamma$ as in (\ref{decompo}) for $R>1$ chosen so that Lemma \ref{ahlforsAk} is valid. Pick $\gamma \in \Gamma$ so that $|\gamma|>2R$ and let $n\geq 2$ be the unique integer  so that $\gamma \in S^{\Gamma}_{n,R}$.
Then using the decomposition (\ref{partition}) we obtain
\begin{align*}
\phi_{s}(\gamma)=e^{-s\alpha |\gamma|}\sum_{k=1}^{n}\int_{A_{k,R}(\gamma)}e^{2s\alpha (\xi,\gamma o)_{o}}d\nu_{o}(\xi).
\end{align*}
If $\xi \in A_{k,R}(\gamma)$ then $kR \leq (\xi,\gamma o)_{o}\leq k R+R$ for all $k\in \{1,\dots, n\}$. Therefore: 
\begin{align*}
 e^{-s\alpha |\gamma|}\sum_{k=1}^{n}\nu_{o}(A_{k,R}(\gamma))e^{2s\alpha kR}\leq \phi_{s}(\gamma)&\leq  e^{-s\alpha |\gamma|}e^{2s\alpha R}\sum_{k=1}^{n}\nu_{o}(A_{k,R}(\gamma))e^{2s\alpha kR}.
\end{align*}
Then, since $\nu_{o}$ is $\alpha$-Ahlfors regular Lemma \ref{ahlforsAk} implies that there exists $C>0$ so that:
\begin{align*}
\phi_{s}(\gamma)&\leq Ce^{-s\alpha  |\gamma| }\sum_{k=0}^{n-1}e^{(2s-1)\alpha kR}\\
& \leq Ce^{-s\alpha |\gamma|} \frac{1-e^{(2s-1)\alpha  nR}}{1-e^{(2s-1)\alpha R}},
\end{align*}

and the lower inequality reads as follows:

$$\phi_{s}(\gamma)\geq C^{-1}e^{-s\alpha |\gamma|} \frac{1-e^{(2s-1)\alpha  nR}}{1-e^{(2s-1)\alpha R}}.$$

Assume now that $s<\frac{1}{2}.$ Observe that $$ 1-e^{(2s-1)\alpha } \leq 1-e^{(2s-1)\alpha R}\leq R(1-e^{(2s-1)\alpha }),$$ with $R>1$. Moreover, since $\gamma$  in $S^{\Gamma}_{n,R}$ satisfies $ nR \leq |\gamma|<(n+1)R$, there exists $C>0$  such that we have
\begin{align*}
\phi_{s}(\gamma)&\leq  C\times \frac{\big(e^{-s\alpha |\gamma|}-e^{(s-1)\alpha |\gamma |}\big)}{1-e^{(2s-1)\alpha }} \\
& =C\times \frac{\bigg(2\sinh  \big(( \frac{1}{2}-s) \alpha |\gamma | \big)\bigg)}{1-e^{(2s-1)\alpha}} \exp\bigg({-\frac{1}{2}\alpha |\gamma|}\bigg).\\
\end{align*}

Since the above inequality holds for $\gamma \in S^{\Gamma}_{n,R}$ for $n\geq 2$, we may find an another constant $C>0$ so that for all $\gamma \in \Gamma$ and for all $s\in ]0,\frac{1}{2}[$ we have: 
\begin{align}\label{upperbound}
\phi_{s}(\gamma)&\leq C\bigg( \frac{2 \sinh  \big(( \frac{1}{2}-s) \alpha |\gamma |\big)}{1-e^{(2s-1)\alpha }}  +1\bigg) \exp\bigg({-\frac{1}{2}\alpha |\gamma|}\bigg).
\end{align}

The left hand inequality reads as follows:   
 \begin{align*}
\phi_{s}(\gamma)\geq C^{-1} \bigg(\frac{ 2\sinh\big( (\frac{1}{2}-s)\alpha |\gamma | \big) }{1-e^{(2s-1)\alpha }}+1\bigg) \exp\bigg({-\frac{1}{2}\alpha |\gamma|} \bigg),
\end{align*}
This lower bound implies Item (2) directly for $s<0$, and use $\phi_{s}(\gamma)=\phi_{1-s}(\gamma^{-1})$ for the case $s>1$.\\

For  $s>\frac{1}{2}$ using $\phi_{s}(\gamma)=\phi_{1-s}(\gamma^{-1})$ and $|\gamma|=|\gamma^{-1}|$, we have for all $\gamma \in \Gamma$
\begin{align*}
 C^{-1}  \bigg(\omega_{s-\frac{1}{2}}(|\gamma|)+1\bigg) \exp\bigg({-\frac{1}{2}\alpha |\gamma|} \bigg) \leq\phi_{s}(\gamma)\leq C \bigg(\omega_{s-\frac{1}{2}}(|\gamma|)+1\bigg) \exp\bigg({-\frac{1}{2}\alpha |\gamma|} \bigg).
\end{align*}
And the proof of Item (3) is complete.\\

\end{proof}

	\subsubsection{Definitions in terms of matrix coefficients}
In this subsection, we give several equivalent definitions of spectral inequalities, on spheres in terms of matrix coefficients.

Given $\Gamma$ a discrete group of isometries of a metric space $(X,d)$, we endow $\Gamma$ with a length function defined as $|\gamma|:=d(\gamma o, o)$ associated with some base point $o\in X$. For any $R>0$, consider the decomposition of $\Gamma$ in spheres as $\Gamma:=\coprod_{n\geq 1} S^{\Gamma}_{n,R}$. 
Let $\pi:\Gamma \rightarrow \mathbb{B}(\mathcal {H})$ be a representation of $\Gamma$ on $\mathcal{H}$.
The following lemma will be useful:

\begin{prop}\label{RDitem}
Let $\sigma$ be in $[0,1]$ and $R>0$.
The following are equivalent:
\begin{enumerate}
\item There exist $C,d>0$ so that for all finitely supported functions $f$  we have $\|\pi(f)\|_{op}\leq C \|f\|_{H^{d}_{\sigma}}.$
\item There exist $C,d>0$ such that for all finitely supported functions $f$ on a sphere of radius $S^{\Gamma}_{n,R}$ and for all $v,w\in \mathcal{H}$ we have 
$$|\langle \pi(f)v,w\rangle |\leq C(1+\omega_{\sigma}(nR))^{d} \|f\|_{2} \|v\| \|w\|.$$
\end{enumerate}
\end{prop}



\begin{proof}
We prove first $(1)$ implies $(2)$.\\
First of all observe that there exists $C>0$ so that  for all non-negative integers $n$ and for  all $\gamma \in S^{\Gamma}_{n,R}$
\begin{align}\label{controleng}
 C^{-1}(1+\omega_{\sigma}(nR))\leq 
 (1+\omega_{\sigma}(|\gamma|))\leq C (1+\omega_{\sigma}(nR)).
 \end{align}
Hence, if $f$ is supported on $S^{\Gamma}_{n,R}$ we control the norm of $f$ as
\begin{equation}\label{control}
C^{-1} (1+\omega_{\sigma}(nR))^{d}\|f\|_{2} \leq 
\|f\|_{H^{d}_{\sigma}}\leq C(1+\omega_{\sigma}(nR))^{d} \|f\|_{2},
\end{equation}
for some $C>0$. Using the right hand side of $(\ref{control})$, we easily see that  $(1)$ implies $(2)$ for the same $d$.\\
We prove now $(2)$ implies $(1)$: Choose $d'>0$ so that $C_{d'}:=\big(\sum^{+\infty}_{n=0}(1+\omega_{\sigma}(nR))^{-2d'}\big)^{\frac{1}{2}}<\infty$ (such $d'>0$ exists by Definition  (\ref{defomeg}) of $\omega_{\sigma}$).
 Since $f$ is finitely supported there exists an integer 
$N$ such that $f$ can be written as $f=\sum^{N}_{n=0}f_{n}$ with $f_{n}=\sum_{\gamma \in S^{\Gamma}_{n,R}}f(\gamma)\delta_{\gamma}$ supported on $S^{\Gamma}_{n,R}$ where $\delta_{\gamma}$ denotes the unit Dirac mass centered at $\gamma$. Observe that the left hand side of inequalities (\ref{controleng}) imply that there exists $C>0$ such that for any $d>0:$
\begin{equation}\label{inegnorm}
C^{-1}\bigg(\sum_{n=0}^{N}(1+\omega_{\sigma}(nR))^{2d}\|f_{n}\|^{2}_{2}\bigg)\leq  \|f\|^{2}_{H^{d}_{\sigma}} 
\end{equation}

The operator norm of $\pi(f)$ satisfies 
\begin{align*}
\|\pi(f)\|_{op}&\leq \sum^{N}_{n=0} \|\pi(f_{n})\|_{op}\\
&\leq  \sum^{N}_{n=0} (1+\omega_{\sigma}(nR))^{d}\|f_{n}\|_{2}\\
&\leq C_{d'}\bigg(\sum^{N}_{n=0} (1+\omega_{\sigma}(nR))^{2(d+d')}\|f_{n}\|^{2}_{2}\bigg)^{\frac{1}{2}}\\
&\leq  C_{d'}\cdot C  \|f\|_{H_{\sigma}^{d+d'}}.
\end{align*}
where the last inequality follows from $(\ref{inegnorm})$ and
 the proof is complete.
\end{proof}

\subsubsection{The theorem is optimal}
Let $\Gamma$ be a non-elementary discrete groups acting by isometries on $(X,d)$, a roughly geodesic, $\epsilon$-good, $\delta$-hyperbolic space.  Consider the representations defined in (\ref{repre}) associated with $(\partial X,\nu_{o})$ where $\nu_{o}$ is the Patterson-Sullivan measure associated with a base point $o$. In the case of hyperbolic groups, the decay of the spherical functions established in Section \ref{sphericalfunctions} provides informations on spectral inequalities. Indeed, the following proposition is a proof of Remark \ref{optimal}. 
\begin{prop}\label{inegl}
There exist $R>0$ and a positive function $f$ on $\Gamma$ and $v,w\in L^{2}(\partial X, \nu_{o})$ such that for all integers $n$ and for all $s\in[0,1]$ $$ \sum_{\gamma \in S^{\Gamma}_{n,R}} f(\gamma) \langle \pi_{s}(\gamma) v,w \rangle\geq C (1+\omega_{|s-\frac{1}{2}|}(nR))^{d} \|v\|_{2} \|w\|_{2}.$$

\end{prop}
\begin{proof}
Pick $R$ so that Lemma \ref{ahlforsAk} holds. Now, consider 
\begin{align*}
f:\gamma \in \Gamma \mapsto \langle \pi_{s}(\gamma) \textbf{1}_{\partial X},\textbf{1}_{\partial X} \rangle \in \mathbb{R}^{+} \mbox{ and } v=w=\textbf{1}_{\partial X}.
\end{align*} 
Thus, since the group is hyperbolic, the Ahlfors regularity of the Patterson-Sullivan measures hold and estimates (\ref{volumegrowth}) combined with lower bound of Item (3) of Proposition \ref{HCHestims}  conclude the proof.
\end{proof}

\subsection{Reduction to positivity, dense subsets and spheres}
In this section, assume that $\Gamma$ denotes any discrete group.\\
Consider a representation $\pi:\Gamma \rightarrow \mathbb{B}(\mathcal{H})$.
Assume that $\mathcal{H}=L^{2}(X,m)$ for some measure space $(X,m)$. The cone of positive functions of $\mathcal{H}$ is denoted by $\mathcal{H}^{+}$. We say that a representation $\pi$ is positive if $\pi$ preserves the cone of positive functions. Observe that $\pi_{s}$ (for $s\in \mathbb{R}$) defined in (\ref{repre}) is positive. \\Let $\mathcal{F}\subset \mathcal{H}$ be a dense subset. We state a useful lemma, reducing spectral inequalities only to positive functions $f$ on $\Gamma$ and vectors of matrix coefficients in the positive cone of a dense subspace $\mathcal{F} \cap \mathcal{H}^{+}$.   
\begin{lemma}\label{reduction}
Assume that $\pi$ is a positive representation of $\Gamma$. Assume there exists a nonnegative integer $n_{0}$ such that for any $R>0$, for all $n\geq n_{0}$,
 for all positive finitely functions $f$ on $S^{\Gamma}_{n,R}$ and for all $v,w\in \mathcal{F} \cap \mathcal{H}^{+}$ we have:
$$\sum_{\gamma \in S^{\Gamma}_{n,R}}f(\gamma)\langle \pi(\gamma)v,w\rangle \leq C (1+\omega_{\sigma}(nR))^{d}\|f\|_{2} \|v\|_{2}\|w\|_{2}.$$

Then there exist $C,d>0$ such that for all $f$ finitely supported we have: $$\|\pi(f)\|_{op}\leq C (1+\omega_{\sigma}(nR))^{d}\|f\|_{2}.$$
\end{lemma}

The proof is easy and left to the reader.

\section{A Dense subset of $L^{2}(\partial X,\nu_{o})$}\label{sec5}
\subsection{Notation}
Let $(X,d)$ be a proper roughly geodesic, $\epsilon$-good, $\delta$-hyperbolic space. We assume that there exists a discrete group of isometries of $(X,d)$
denoted by $\Lambda$  acting properly and cocompactly on $(X,d)$ (a priori different from a group $\Gamma$  for which we shall establish property RD). The elements of $\Lambda$ will be denoted by $g\in \Lambda $ or $h\in \Lambda$.
Let $\nu=(\nu_{x})_{x\in X}$ be the Patterson-Sullivan measure associated with $\Lambda$ of dimension $\alpha$. We denote by $\rho$ the diameter of a fundamental domain of the action of $\Lambda$ on $(X,d)$ so that one can write $\cup_{g\in \Lambda} B(go, \rho)=X$, where $o$ is a basepoint in $X$. 

In this section, we provide the construction of a sequence of finite dimensional  subspaces  of $L^{2}(\partial X,\nu_{o})$, denoted by $E^{\Lambda}_{N}(\partial X)$, depending strongly on $\Lambda$ such that for all $v\in L^{2}(\partial X,\nu_{o})$ there is a sequence of $v_{N}\in E^{\Lambda}_{N}(\partial X)$ satisfying $\|v-v_N\|_{2}\to 0$ as $N\to +\infty$.\\
Thus, the dense subset we shall consider is defined as 
\begin{align}\label{densesub}
\mathcal{F}:= \cup_{N\geq 0} E^{\Lambda}_{N}(\partial X).
\end{align}

\subsection{Covering, multiplicity, shadows}
The following lemma is very useful. We shall mention that it has been already used in several situations, see for example \cite[Corollary 4.2]{BM} and \cite[Lemma 4.2]{Ga}. This is the cornerstone to construct a suitable sequence of subspaces on the boundary.\\
We consider the spheres $S^{\Lambda}_{N,R}$ of $\Lambda$ for all non-negative integers $N$ and for $R>0$. 
\begin{lemma}\label{fondam} There exist $R,r>0$ so that for all $N$ large enough we have

\begin{enumerate}
\item $\cup_{g\in S^{\Lambda}_{N,R}}O_{r}(o,go)=\partial X.$ 
\item There exists an integer $\frak{m}$ such that for all $N$ and for all $g\in S^{\Lambda}_{N,R}$ the cardinal of the set $\{h\in S^{\Lambda}_{N,R}| O_{r}(o,ho)\cap O_{r}(o,go)\neq \varnothing \}$ is bounded by $\frak{m}$.
\end{enumerate}
\end{lemma}

\begin{proof}
We prove Item (1): Let $R$ be larger than $2(C_{X}+\rho)$, where the constant $C_{X}$ comes from the roughly geodesics and let $C_{X}+\rho<\tau<R-C_{X}-\rho$. Consider $\xi \in \partial X$ and let a $r_{o}$ be a roughly isometry representing a roughly geodesic starting at $o$ and ending at $\xi$. Consider the point $x=r_{o}(NR+\tau)$ thus satisfying $NR+\tau-C_{X} \leq d(o,x)\leq NR+\tau+C_{X}$.  Since $\Lambda$ acts cocompactly on $X$, there exists $g\in \Lambda$ so that $d(go,x)\leq \rho$.\\
Then $NR+\tau-C_{X}-\rho<d(go,o)\leq NR+\tau+C_{X}+\rho.$ The choice of $\tau$ ensures that $g\in S^{\Lambda}_{N,R}$. 
Observe that 
\begin{align*}
 (\xi,x)_{o}&=\lim_{n\to +\infty}(r_{o}(n),r_{o}(NR+\tau ))_{o}\\
 &= \lim_{n\to +\infty} \frac{1}{2}( d(r_{o}(n),o)+d(o,r_{o}(NR+\tau ))- d(r_{o}(n),r_{o}(NR+\tau )) )\\
 &\geq NR+\tau +C+\rho-\frac{3}{2}C_{X}-C_{X}-\rho\\
 &\geq d(o,g o)-\frac{1}{2}C_{X}-\rho.
\end{align*}

It follows that:
\begin{align*}
(\xi,go)_{o}&\geq \min{ \{ (\xi,x)_{o},(x,g o)_{o} \} }-\delta\\
&\geq \min{ \{d(o,g o)-\frac{1}{2}C_{X}-\rho,\frac{1}{2}(d(o,x)+d(o,go)-d(x,go)) \} }-\delta\\
&\geq  \min{ \{ d(o,g o)-\frac{1}{2}C_{X}-\rho,d(o,go)-C_{X}-\rho \} }-\delta\\
&\geq d(o,go)-C_{X}-\rho-\delta.
\end{align*}
We set $r:=C_{X}+\rho+\delta.$ Hence $\xi \in O_{r}(o,go)$.\\

We prove Item (2): Fix now some $g\in S^{\Lambda}_{N,R}$. Let $h\in S^{\Lambda}_{N,R}$ such that $O_{r}(o,go)\cap O_{r}(o,ho)\neq \varnothing$. Pick $\eta\in O_{r}(o,go)\cap O_{r}(o,ho)$ and so $$(go,ho)_{o}\geq \min \{(go,\eta)_{o},(\eta,ho)_{o}\}-\delta\geq NR-r-\delta.$$
Hence, $$d(go,ho)=d(o,go)+d(o,ho)-2(go,ho)_{o}\leq 2(N+1)R-2NR+2r+2\delta=2(R+r+\delta).$$ 
Set $K:=2(R+r+\delta)$. Thus $\{h\in S^{\Lambda}_{N,R}| O_{r}(o,ho)\cap O_{r}(o,go)\neq \varnothing \}\subset B_{X}(go,K)$. Then $$|\{h\in S^{\Lambda}_{N,R}| O_{r}(o,ho)\cap O_{r}(o,go)\neq \varnothing \}|\leq |\Lambda \cap B_{X}(o,K)|.$$
Set $\frak{m}=|\Lambda\cap B_{X}(o,K)|$ to finish the proof, since $\Lambda$ is discrete.
\end{proof}

It is worth noting that, using Ahlfors regularity of the measure $\nu_{o}$ we deduce 

from Item (1) of Lemma \ref{fondam}: 
\begin{equation}
 |S^{\Lambda}_{N,R}|\leq C_{\nu} e^{\alpha(r+\delta)}\times  e^{\alpha NR}.
\end{equation}


\subsection{Vitali's covering type Lemma}
We use Lemma \ref{fondam} to construct a very useful sequence of finite dimensional subspaces of $L^{2}(\partial X,\nu_{o})$. But before giving the definition of these subspaces we shall use a Vitali's covering argument for the construction. 
\begin{lemma}\label{vitali}
Let $R$ and $r$ as in Lemma \ref{fondam}. For all non-negative integers $N$ large enough, there exists a non empty set $S^{\Lambda*}_{N,R}\subset S^{\Lambda}_{N,R}$ and Borel subsets $\big(Q_{r}(o, g o )\big)_{g\in S^{\Lambda*}_{N,R}}$ of the boundary such that 
\begin{enumerate}
\item $\partial X = \coprod_{g\in S^{\Lambda *}_{N,R}}Q_{r}(o, go) $.,\\
\item $Q_{r}(o, go)\cap Q_{r}(o, ho)=\varnothing$ with $g\neq h\in S^{\Lambda *}_{N,R}$, \\
\item $O_{r}(o, go)\subset Q_{r}(o, go)\subset O_{r'}(o, go) $ for some $r'>r$ for all $g\in S^{\Lambda*}_{N,R}$,\\
\item There exists a constant $C>0$ such that for all integers $N$ we have:
$$C^{-1}e^{\alpha NR} \leq |S^{\Lambda*}_{N,R}|\leq Ce^{\alpha NR}.$$
\end{enumerate}
\end{lemma}
\begin{proof}
From Lemma \ref{fondam} we have $\cup_{g\in S^{\Lambda}_{N,R}}O_{r}(o,go)=\partial X$.
Lemma \ref{ombre} implies that $O_{r}(o,go)\subset B(\theta_{o}^{go},e^{-\epsilon(|g|-r-\delta)})$. Since $g\in S^{\Lambda}_{N,R}$, we have $O_{r}(o,go)\subset B(\theta_{o}^{go},e^{-\epsilon(NR-r-\delta)})$. Define the ball associated with $g\in S^{\Lambda}_{N,R}$ as $B_{g}:= B(\theta_{o}^{go},r_{N})$ with $r_{N}:=e^{-\epsilon(NR-r-\delta)}$. Hence $\cup_{g\in S^{\Lambda}_{N,R}}B_{g}=\partial X$. Denote by $B^{*}_{g}$ the ball $B(\theta_{o}^{go},5 r_{N})$. By Vitali's covering lemma there exists a set $S^{\Lambda*}_{N,R}\subset S^{\Lambda}_{N,R}$ so that 
\begin{equation}
\cup_{g\in S^{\Lambda *}_{N,R}}B^{*}_{g}=\partial X,
\end{equation} and such that 
\begin{equation}
B_{g}\cap B_{h}=\varnothing, \mbox{for  all } g\neq h\in S^{\Lambda*}_{N,R}.\end{equation} Therefore one can construct a family of Borel subsets $Q_{r}(o,g o)$ for $g\in S^{\Lambda*}_{N,R}$ such that Item $(1),(2)$ and $(3)$ of Lemma \ref{vitali} are satisfied (see for example \cite[Lemma 2, p15]{St}): enumerate the elements $g_{k}$ of $S^{\Lambda*}_{N,R}$ for $k\in I_{N}$ with index set $I_{N}:=\{1,\dots,|S^{\Lambda *}_{N,R}|\}$. Then consider $B_{g_{k}}$ for $k\in I_{N}$ and define by induction:
$$Q_{r}(o,g_{k} o):=B^{*}_{g_{k}}\cap \big(\partial X\backslash \cup_{j<k}Q_{r}(x,g_{j} x) \big) \cap \big(\partial X\backslash  \cup_{j>k} B_{g_{j}} \big).$$\\

For Item (4), observe first that there exists $C>0$ such that for all $g\in S^{\Lambda *}_{N,R}$ we have 
\begin{equation}\label{ahlfors}
C^{-1}e^{-\alpha NR }\leq \nu_{o}(Q_{r}(o, g o))\leq Ce^{-\alpha NR }.
\end{equation}

 Then, write $\partial X=\coprod_{g\in S^{\Lambda*}_{N,R}}Q_{g}$ as a disjoint union. By taking the measure of two members of this equality we obtain:
\begin{align*}\label{counting}
\nu_{o}(\partial X)&=\nu_{o}\bigg(\coprod_{ g\in S^{\Lambda*}_{N,R} }Q_{r}(o, g o) \bigg)=\sum_{g\in S^{\Lambda*}_{N,R} }\nu_{o}\big(Q_{r}(o, g o)\big),\\
\end{align*}
and Inequalities (\ref{ahlfors}) imply: 
\begin{equation}\label{countingsphere}
 C^{-1}e^{\alpha N R} \leq |S^{\Lambda*}_{N,R}|\leq C e^{\alpha N R}.
 \end{equation}

\end{proof}
In the following, we denote $Q_{r}(o, go)$ by $Q_{g}$, for $g\in S^{*\Lambda}_{N,R}$ 
\subsection{Construction of a dense subspace}

Define the finite dimensional subspaces of $E^{\Lambda}_{N}(\partial X)\subset L^{2}(\partial X,\nu_{o})$ as the subspaces generated by 
\begin{equation}\label{sousespace}
 E^{\Lambda}_{N}(\partial X):=\mbox{Span}\{ \textbf{1}_{Q_{g} }\mbox{ with } g\in S^{\Lambda*}_{N,R} \},
\end{equation}
where $\textbf{1}_{Q_{g} }:= \textbf{1}_{Q_{r}(o,g o)}.$\\

We have the following proposition:

\begin{prop}\label{density}
For all $v\in L^{2}(\partial X,\nu_{o})$ there exists a sequence $v_{N}\in E^{\Lambda}_{N}(\partial X,\mu_{x})$ so that $\|v-v_{N}\|_{2}\to 0$ as $N \to +\infty$.
\end{prop}

\begin{proof}
First observe that the space of Lipschitz functions on the boundary is dense in $L^{2}(\partial X, \nu_{o})$. 

Then, define the projection $p_{N}$ for a non-negative integer $N$ as
\begin{equation}
p_{N}:v\in L^{2}(\partial X, \nu_{o})\mapsto \sum_{g\in S^{\Lambda}_{N,R}}\frac{1}{\nu_{o}(Q_{g})}\langle v,  \textbf{1}_{Q_{g}}\rangle \textbf{1}_{Q_{g}}\in E^{\Lambda}_{N}(\partial X).
\end{equation}
 If $v$ is a L-Lipschitz function, then $p_{N}(v)\in E^{\Lambda}_{N}(\partial X)$ for all $N$ and we have $\|v-v_{N}\|_{2}\to 0$ as $N \to +\infty$.

Indeed for $v$ a $L$-Lipschitz function on the boundary we have for all $ h\in S^{\Lambda*}_{N,R}$ and for all $\xi\in Q_{h}$: 
\begin{align*}
\bigg|v(\xi)\textbf{1}_{Q_{h}}(\xi)-\bigg(\frac{1}{\nu_{o}(Q_{h})}\int_{Q_{h}}v(\eta)d\nu_{o}(\eta)\bigg)\textbf{1}_{Q_{h}}(\xi)\bigg|&=\bigg|\frac{1}{\nu_{o}(Q_{h})}\int_{Q_{h}}v(\xi)-v(\eta)d\nu_{o}(\eta)\bigg|\\ 
& \leq \frac{1}{\nu_{o}(Q_{h})}\int_{Q_{h}}|v(\xi)-v(\eta)|d\nu_{o}(\eta)\\
& \leq \frac{1}{\nu_{o}(Q_{h})}\int_{Q_{h}}L |d_{o}(\xi)-d_{o}(\eta)|d\nu_{o}(\eta)\\
 &\leq L \sup_{h\in S^{\Lambda*}_{N,R}}\mbox{Diam}(Q_{h}).
\end{align*}
 Moreover for $v\in L^{2}(\partial X,\mu)$:
 \begin{align*}
 \| v-p_{N}(v)\|_{2}^{2} &=\int_{\partial X}\bigg[v(\xi)-\sum_{g \in S^{*  \Lambda}_{N,R}}\bigg(\frac{1}{\mu(Q_{g})}\int_{Q_{g}}v (\eta)d\mu(\eta)\bigg)\textbf{1}_{Q_{g}}(\xi)\bigg]^{2}d\nu_{o}(\xi)\\
 &=\sum_{h\in S^{\Lambda*}_{N,R}}\int_{\partial X}\bigg[v(\xi)-\sum_{g}\bigg(\frac{1}{\nu_{o}(Q_{g})}\int_{Q_{g}}v(\eta)d\nu_{o}(\eta)\bigg) \textbf{1}_{Q_{g}}(\xi)\bigg]^{2}\textbf{1}_{Q_{h}}(\xi)d\nu_{o}(\xi)\\
  &=\sum_{h \in S^{*  \Lambda}_{N,R}}\int_{Q_{h}}\bigg[v(\xi) \textbf{1}_{Q_h}(\xi)-\bigg(\frac{1}{\nu_{o}(Q_{h})}\int_{Q_{h}}v(\eta)d\mu(\eta)\bigg) \textbf{1}_{Q_h}(\xi)\bigg]^{2}d\nu_{o}(\xi)\\
 &\leq \big(L \sup_{h\in S^{*  \Lambda}_{N,R} }\mbox{Diam}(Q_{h})\big)^{2}\sum_{h \in S^{*  \Lambda}_{N,R}}\int_{Q_{h}}\textbf{1}_{\partial X}d\nu_{o}(\xi)\\
 &=\big(L \sup_{h\in S^{*  \Lambda}_{N,R} }\mbox{Diam}(Q_{h})\big)^{2}\to 0,
 \end{align*}
 as $N$ goes to infinity, and the proof is done. Note that the inequality in the above computation follows from the previous inequality just above in the proof.

\end{proof}

In other words, the set $\mathcal{F}$ defined in (\ref{densesub}) is dense in $L^{2}(\partial X,\nu_{o})$.
\subsection{$L^{2}$-norm }
If $v\in E^{\Lambda}_{N}(\partial X)$, write $v=\sum_{g\in S^{\Lambda*}_{N,R}}d_{g}\textbf{1}_{Q_g}$ where $d_{g}$ are a priori complex numbers. Observe that the $L^{2}$-norm of $v$ is given by
\begin{equation*}
\|v\|_{2}:=\big(\sum_{g\in S^{\Lambda *}_{N,R}}|d_{g}|^{2}\nu_{o}(Q_{g}) \big)^{\frac{1}{2}}.
\end{equation*}
By Inequalities (\ref{ahlfors}), there exists a constant $C>0$ such that:
\begin{align}\label{norm2}
\frac{C ^{-1}}{|S^{\Lambda*}_{N,R}|^{\frac{1}{2}}}\big(\sum_{g\in S^{\Lambda*}_{N,R}}|d_{g}|^{2}\big)^{\frac{1}{2}} \leq  \|v\|_{2} \leq    \frac{C}{|S^{\Lambda*}_{N,R}|^{\frac{1}{2}}}  \big(\sum_{g\in S^{\Lambda *}_{N,R}}|d_{g}|^{2}\big)^{\frac{1}{2}}.
\end{align}

\section{Use of cocompacity}\label{sec6}

In this section, let $\Gamma$ be a discrete group of $(X,d)$ a proper roughly geodesic $\epsilon$-good $\delta$-hyperbolic space. We assume that $\Gamma$ acts properly and cocompactly on $(X,d)$. We fix a fundamental domain $\mathcal{D}$ for the action of $\Gamma $ on $(X,d)$, containing a base point $o$, relatively compact of diameter $\rho>0$. We shall think about $\Gamma$ as the group for which we want to prove property RD.
\subsection{Cocompacity implies  uniformly bounded multiplicity}

Fix $R>0$.
Let $n$ be a non-negative integer, $k$ another integer in $\{1,\dots,n\}$ and define for each pair of open balls $B(\xi,r_{N})\times B(\eta,r_{N})\subset \partial X \times \partial X$ centered at $\xi,\eta\in \partial X$ of radius $r_{N}=e^{-\epsilon RN}$ with $N>n$,    the set 
\begin{equation}
S^{\Gamma}_{k,n-k,N}(\xi,\eta):=\{ \gamma \in S^{\Gamma}_{n,R}| A_{k,R}(\gamma )\times \gamma ^{-1} A_{k,R}(\gamma  )\cap B(\xi,r_{N})\times B(\eta,r_{N})\neq \varnothing \}\subset S^{\Gamma}_{n,R}.
\end{equation}
The following definition is inspired by the work of Bader-Muchnik in \cite[Definition 4.1]{BM}.

\begin{defi} 
Let $N>n$ with $n\geq 1$  and $k\in \{1,\dots ,n\}$ where $N,n,k$ are non-negative integers.
We say that $(A_{k,R}(\gamma)\times \gamma^{-1} A_{k,R}(\gamma))_{\gamma \in S^{\Gamma}_{n,R}}$ is a $(k,n,N)$-sampling for $\partial X \times \partial X$ of multiplicity $\frak{m}\in \mathbb{N}$, if for all $\xi,\eta\in \partial X \times \partial X$
we have 
$$ |S^{\Gamma}_{k,n-k,N}(\xi,\eta)|\leq \frak{m}.$$
\end{defi}

The fundamental result to prove property RD, as well as the other spectral inequalities dealing with $\pi_{s}$ for $s\neq\frac{1}{2}$, is the following proposition providing an uniform sampling:

\begin{prop}\label{mgeneral} If $\Gamma$ acts cocompactly then there exists an integer $\frak{m}$ such that for all integers $N>n\geq 1$ and for all integers $k\in \{ 1,\dots,n\}$ we have  that $(A_{k,R}(\gamma)\times \gamma^{-1} A_{k,R}(\gamma))_{\gamma \in S^{\Gamma}_{n,R}}$ is a $(k,n,N)$-sampling for $\partial X \times \partial X$ of multiplicity $\frak{m}\in \mathbb{N}$.
\end{prop}

\begin{proof}
We proceed in two steps.

\textbf{Step 1:} 
There exist two positive real numbers $\rho_{1},\rho_{2} >0$ such that for all $N>n$ and $k$ be integer in $\{1,\dots,n\}$, and for all $\xi,\eta \in \partial X$, if $\gamma\in S^{\Gamma}_{k,n-k,N}(\xi,\eta)$ we have $$\gamma^{-1} B_{X}(\xi_{k},\rho_{1})\cap B_{X}(\eta_{n-k},\rho_{2})\neq \varnothing.$$

\begin{proof}[Proof of Step 1]

Let $\gamma \in S^{\Gamma}_{n,R}$ such that $A_{k,R}(\gamma)\times  \gamma ^{-1}A_{k,R}(\gamma) \cap  B(\xi,r_{N})\times B(\eta,r_{N})  \neq \varnothing$. Consider the point $\xi_{k}\in X$ satisfying $ kR \leq d(x,\xi_{k})\leq (k+1)R$ and the point $\eta_{n-k}\in X$ satisfying $ (n-k)R \leq d(x,\eta_{n-k})\leq (n-k+1)R$.\\
Then, consider $\gamma_{k}$ in $S^{\Gamma}_{k,R}$ (recall the definition \ref{zkg}).\\

\textbf{Claim 1: There exists $\rho_{1}>0$ so that $d(\gamma_{k},\xi_{k})\leq \rho_{1}$. }\\
\begin{proof}[Proof of Claim 1]
Write $d(\gamma_{k},\xi_{k})=d(o,\gamma_{k})+d(o,\xi_{k})-2(\gamma_{k},\xi_{k})_{o}$. We shall find a lower bound for the term $(\gamma_{k},\xi_{k})_{o}$.
 
Observe that Lemma \ref{roughlygeo1} implies there exists $C>0$ such that for all $k\in \{1,\dots,n\}$
\begin{align*}
(\gamma_{k},\xi_{k})_{o}&\geq \min\{(\gamma_{k},\gamma o)_{o},(\gamma o,\xi_{k})_{o}\}-\delta\\
&\geq \min\{kR-C,(\gamma o,\xi_{k})_{o}\}-\delta.\\
\end{align*}

Pick now $\xi'\in A_{k,R}(\gamma)\cap B(\xi,r_{N})$. Then several uses of the hyperbolic inequality (\ref{hyp}) leads to:
\begin{align*}
(\gamma o, \xi_{k})_{o}&\geq \min{\{ (\gamma o,\xi'_{k})_{o},(\xi'_{k},\xi_{k})_{o}\}}-\delta\\
&\geq \min{\{ (\gamma o,\xi'_{k})_{o},(\xi'_{k},\xi')_{o}, (\xi',\xi_{k})_{o}\}}-2\delta\\
&\geq \min{\{ (\gamma o,\xi'_{k})_{o},(\xi'_{k},\xi')_{o}, (\xi',\xi)_{o},(\xi,\xi_{k})_{o}\}}-3\delta\\
&\geq \min{\{ (\gamma o,\xi)_{o}, (\xi'_{k},\xi')_{o}, (\xi',\xi)_{o},(\xi,\xi_{k})_{o}, ,  (\xi, \xi'_{k})_{o}\}}-4\delta\\
&\geq \min{\{(\gamma o,\xi')_{o}, (\xi',\xi)_{o}, (\xi'_{k},\xi')_{o}, (\xi,\xi_{k})_{o} \}}-5\delta.
 \end{align*}
Since $\xi,\xi' \in  B(\xi,r_{N})$, we have $(\xi',\xi)_{o}\geq (N+1)R>nR$. Besides, Lemma \ref{ineqzkend} implies $$ kR+\frac{1}{2}C_{X}\leq (\xi'_{k},\xi')_{o},(\xi_{k},\xi)_{o}\leq kR+\frac{7}{2}C_{X}\leq (N+1)R<(\xi,\xi')_{o},$$ and since $\xi'\in A_{k,R}(\gamma)$ we have by definition $(\gamma o,\xi')_{o}\geq kR$. Finally, we obtain $$(\gamma o, \xi_{k})_{o}\geq kR-5\delta.$$

 Hence, 
 \begin{align}
 (\gamma_{k},\xi_{k})_{o}\geq kR-C-6\delta.
 \end{align}
So, since $d(\gamma_{k},\xi_{k}) )=d(o,\gamma_{k})+d(o,\xi_{k})-2(\gamma_{k},\xi_{k})_{o}$ we have, by (\ref{distancezk})
\begin{align*}
d(\gamma_{k},\xi_{k}) )& \leq  2kR+6C_{X} -2 (kR-C-6\delta)\\
&\leq 6C_{X}+2C+12\delta.
\end{align*}
Set $\rho_{1}:=6C_{X}+2C+12\delta$ to conclude the proof of Claim 1.\\

\end{proof}
\textbf{Claim 2: There exists $\rho_{2}>0$ so that $d(\gamma^{-1}\gamma_{k},\eta_{n-k})\leq \rho_{2}$. }\\

\begin{proof}[Proof of Claim 2]

 Lemma \ref{ineqzkend} implies: 

\begin{align*}
(\gamma^{-1}\gamma_{k},\eta_{n-k})_{o}&\geq \min{\{(\gamma^{-1}\gamma_{k},\eta))_{o},(\eta,\eta_{n-k})_{o}}\}-\delta\\
&\geq \min{\{(\gamma^{-1}\gamma_{k},\eta))_{o},(n-k)R+\frac{1}{2}C_{X} }\}-\delta.
\end{align*}
Observe  that Lemma \ref{controlzkxx} implies that there exists a constant $C>0:$
\begin{align}\label{ineqpass}
(\gamma^{-1}\gamma_{k},\gamma^{-1} o)_{o}&=\frac{1}{2}(d(\gamma o,\gamma_{k})+d(o,\gamma^{-1} o)-d(\gamma_{k} ,o) )\\
&\geq (n-k)R-\frac{1}{2}(R+C).
\end{align}
Pick now $\eta'\in \gamma^{-1}A_{k,R}(\gamma )\cap B(\eta,r_{N})$. Then Lemma \ref{lesAi} implies that $\eta'\in A_{n-k,R}(\gamma^{-1} )\cup  A_{n-k-1,R}(\gamma^{-1} )\cap B(\eta,r_{N})$.

 We have 
\begin{align*}
(\gamma^{-1}\gamma_{k},\eta)_{o}&\geq \min\{(\gamma^{-1}\gamma_{k},\eta')_{o},(\eta',\eta)_{o} \}-\delta \\
&\geq \min\{(\gamma^{-1}\gamma_{k},\gamma^{-1} o)_{o},(\gamma^{-1},\eta')_{o},(\eta',\eta)_{o} \}-2\delta \\
& \geq \min\{(\gamma^{-1}\gamma_{k},\gamma^{-1} o)_{o},(\gamma^{-1},\eta')_{o} ,(\eta',\eta)_{o}\}-2\delta\\
& \geq\min\{(n-k)R-\frac{1}{2}(R+C),(n-1-k)R\}-2\delta,
\end{align*}
where the last inequality follows from the definition of $A_{n-k,R}(\gamma)$, the inequality (\ref{ineqpass}), and the fact that $(\eta,\eta')_{o}\geq NR>(n-k)R$.

 Therefore there exists a constant $C'>0$ such that
$$ (\gamma^{-1}\gamma_{k},\eta_{n-k})_{o}\geq (n-k)R-C'.$$

Lemma \ref{controlzkxx} provides a constant $C>0$ for the first term of the following right hand side equality such that: 
\begin{align*}
d(\gamma^{-1}\gamma_{k},\eta_{n-k})&=d(\gamma^{-1}\gamma_{k},o)+d(o,\eta_{n-k})-2(\gamma^{-1}\gamma_{k},\eta_{n-k})_{o}\\
&\leq (n-k)R+C+(n+1-k)R-2 (n-k)R +2C'\\
&=R+C+2C'.
\end{align*}
Set $\rho_{2}:=R+C+2C'$ to finish the proof.

\end{proof}
This achieves the proof of Step 1.
\end{proof}

\textbf{Step 2:} 
Step 1 implies that there exists $\rho_{1}>0$ and $\rho_{2}>0$ such that for each pair $(\xi,\eta)\in \partial X \times \partial X$ and for all $N>n$ and for all $k\in\{1,\dots,n \}$   we have: $$ S^{\Gamma}_{n,n-k,N}(\xi,\eta)\subset \{\gamma \in S^{\Gamma}_{n,R} | \gamma^{-1} B_{X}(\xi_{k}, \rho_{1}) \cap B_{X}(\eta_{n-k}, \rho_{2})\neq \varnothing\}.$$

Since the action $\Gamma \curvearrowright X$ is cocompact, one can find two elements $x,y$ in a fundamental domain $\mathcal{D}$ containing $o$ of diameter $\rho$, and two elements $\gamma_{1},\gamma_{2}\in \Gamma$ such that $d(\xi_{k},  \gamma_{1}x)<\rho$ and $d(\eta_{n-k},\gamma_{2} y)<\rho$. Thus
\begin{align*}
 S^{\Gamma}_{n,n-k,N}(\xi,\eta)&\subset \{\gamma \in S^{\Gamma}_{n,R} | \gamma^{-1} B_{X}(\gamma_{1} x, \rho_{1}+\rho) \cap B_{X}(\gamma_{2} y, \rho_{2}+\rho)\neq \varnothing\}\\
\end{align*}
Since the counting measure is bi-invariant we have
\begin{align*}
|S^{\Gamma}_{n,n-k,N}(\xi,\eta)|\leq |\{\gamma \in S_{n,R}^{\Gamma} | \gamma^{-1}B_{X}( x, \rho_{1}+\rho) \cap B_{X}(y, \rho_{2}+\rho)\neq \varnothing\}|.
\end{align*}

Eventually, use the fact that the action of $\Gamma$ on $X$ is properly discontinuous to define the non-negative integer:
\begin{equation}\label{integerm}
\frak{m}:=|\{\gamma \in \Gamma | \gamma^{-1} B_{X}( x, \rho_{1}+\rho) \cap B_{X}(y, \rho_{2}+\rho)\neq \varnothing\}|.
\end{equation}
We obtain:
$$|S^{\Gamma}_{n,n-k,N}(\xi,\eta)|\leq \frak{m}.$$

\end{proof}

\subsection{Uniform boundedness on the group}
Let $\Lambda$ be a discrete group of isometries of $(X,d)$, a priori different from $\Gamma$. We assume that $\Lambda$ acts properly and cocompactly on $(X,d)$ so that the results of Section \ref{sec5} hold.   \\

Let $N>n\geq 1$ and $k\in\{1,\dots,n \}$ be non-negative integers. Consider the sets $(Q_{g})_{g\in S^{\Lambda,*}_{N,R}} $ given by Lemma \ref{vitali}. Using Lemma \ref{ombre} we have $Q_{g}\subset B(\theta_{o}(g o),5r_{N}) $
with $r_{N}:=e^{-\epsilon(NR-r-\delta)}$ for all $g\in S^{\Lambda,*}_{N,R}$. By choosing $R>0$ large enough i.e. so that \begin{equation}\label{condR} R>r+\delta+\frac{\log(5)}{\epsilon},\end{equation} we have 
\begin{equation}\label{choiceR}
Q_{g}\subset B(\theta_{o}(g o),e^{\epsilon (N-1)R}).
\end{equation}
\subsubsection{Partitions of $S^{\Lambda *}_{N,R}$}

Let $\gamma $ be in $S^{\Gamma}_{n,R}$ and define 
\begin{equation}\label{notationSn1}
S^{\Lambda}_{N,k}(\gamma):=\{ g\in S^{\Lambda*}_{N,R}| A_{k,R}(\gamma )\cap Q_{g}\neq \varnothing\}\subset S^{\Lambda *}_{N,R}.
\end{equation}
and 
\begin{equation}\label{notationSn2}
\widetilde{S}^{\Lambda}_{N,n-k}(\gamma^{-1}):=\{ h\in S^{\Lambda*}_{N,R}| A_{k,R}(\gamma ) \cap  \gamma Q_{h}\neq \varnothing\}\subset S^{\Lambda *}_{N,R}.
\end{equation}
Note that
$$ \widetilde{S}^{\Lambda}_{N,n-k}(\gamma^{-1})=\{ h\in S^{\Lambda*}_{N,R}| \gamma^{-1}A_{k,R}(\gamma ) \cap  Q_{h}\neq \varnothing\}.
$$
\begin{remark}
The notation  $\widetilde{S}^{\Lambda}_{N,n-k}(\gamma^{-1})$ is legitimate by Lemma \ref{lesAi} implying that: 
\begin{equation}
 \widetilde{S}^{\Lambda}_{N,n-k}(\gamma^{-1})\subset S^{\Lambda}_{N,n-k}(\gamma^{-1} )\cup  S^{\Lambda}_{N,n-k-1}(\gamma^{-1} )
\end{equation}
\end{remark}

 For each $\gamma_{0} \in S^{\Gamma}_{n,R}$ define the set 
\begin{equation}
S^{\Gamma}_{n,n-k,N}(\gamma_{0}):=\{\gamma \in S^{\Gamma}_{n,R}| S^{\Lambda}_{N,k}(\gamma)\times  \widetilde{S}^{\Lambda}_{N,n-k}(\gamma^{-1})\cap  S^{\Lambda}_{N,k}(\gamma_{0} )\times  \widetilde{S}^{\Lambda}_{N,n-k}(\gamma^{-1}_{0} ) \neq \varnothing \}\subset S^{\Gamma}_{n,R}.
\end{equation}
 A consequence of Proposition \ref{mgeneral} is the following:

\begin{prop}\label{mfinal} If $\Gamma$ acts cocompactly then there exists a non-negative integer $\frak{m}$ such that for all non-negative integers $N+1>n$ and for all integers $k\in \{1,\dots,n\}$ and for all $\gamma_{0} \in S^{\Gamma}_{n,R}$: $$|S^{\Gamma}_{k,n-k,N}(\gamma_{0})|\leq \frak{m}.$$
\end{prop}

\begin{proof}
Let $k$ be an integer in $\{ 1,\dots,n\}$. Fix $\gamma_{0} \in S^{\Gamma}_{n,R}$ and consider $\gamma\in S^{\Gamma}_{n,R}$ so that  $S^{\Lambda}_{k,R}(\gamma )\times  \widetilde{S}^{\Lambda}_{n-k,R}(\gamma ^{-1} )\cap S^{\Lambda}_{k,R}(\gamma_{0} )\times  \widetilde{S}^{\Lambda}_{n-k,R}(\gamma^{-1}_{0} ) \neq \varnothing.$ By definition of these sets, there exist $g,h\in S^{\Lambda}_{N}$ so that 
\begin{equation}\label{inter1}
Q_{g}\cap  A_{k,R}(\gamma )\neq\varnothing \mbox{ and }  Q_{g}\cap A_{k,R}(\gamma_{0})\neq\varnothing,
\end{equation}
 as well as 
 \begin{equation}\label{inter2}
 Q_{h}\cap \gamma^{-1}A_{k,R}(\gamma) \neq\varnothing \mbox{ and } Q_{h}\cap \gamma_{0}^{-1}A_{k,R}(\gamma_{0} ) \neq\varnothing.
 \end{equation} 

 By the rigth hand inclusion of Item (3) of Proposition \ref{vitali}, thanks to Lemma \ref{ombre} and by the choice of $R$ in \ref{choiceR}, we have for  $N+1>n$ that $Q_{g}\subset B(\theta_{o}^{go},e^{-\epsilon NR})$ and $Q_{h}\subset B(\theta_{o}^{ho},e^{-\epsilon (N-1)R})$. Therefore 
\begin{align*}
S^{\Gamma}_{k,n-k,N}(\gamma_{0})\subset S^{\Gamma}_{k,n-k,N}(B(\theta_{o}^{go},e^{-\epsilon (N-1)R}),B(\theta_{o}^{ho},e^{-\epsilon (N-1)R})),
\end{align*}
and since $N+1>n$, Proposition \ref{mgeneral} finishes the proof.
\end{proof}

\section{Counting problem for $\Lambda$}\label{sec7}

\subsection{A counting estimate lemma}

Let $\Lambda$ be a discrete group of isometries of $(X,d)$ a proper roughly geodesic, $\epsilon$-good, $\delta$-hyperbolic space. Assume that $\Lambda$ acts properly and cocompactly. Let $\nu_{o}$ be the Patterson-Sullivan measure associate with a base point $o\in X$ of conformal dimension $\alpha$.\\
Let $R>0$. It turns out that the growth of $S^{\Lambda}_{N,R}$ behaves as $e^{\alpha NR}$. Given a Borel subset $U$ of the  boundary (with a frontier of measure zero), the number of elements of the sphere $S^{\Lambda}_{N,R}$  such that the shadows associated to them intersect $U$ are not empty, growths as $\nu_{o}(U)e^{\alpha NR}$. The following lemmas are \emph{uniform} refinement of these counting estimates rather than asymptotic estimates. We shall provide upper bound, with  uniform constants in $N$, of the number of elements of the sphere $S^{\Lambda}_{N,R}$  such that the shadows associated to them intersect $U$. But the counting estimates shall deal with the spheres $S^{\Lambda*}_{N,R}$ rather than the spheres $S^ {\Lambda}_{N,R}$. More precisely for a Borel subset $U\subset \partial X$ define the set
\begin{equation}\label{defU}
W^{\Lambda}_{N,R}(U):=\{ g\in S^{\Lambda*}_{N,R}|U \cap Q_{g}\neq \varnothing\}\subset  S^{\Lambda*}_{N,R}.
\end{equation}
We start first with a variation of the so-called \emph{Sullivan's shadow lemma}. We can find a proof of the following lemma in \cite[Lemma 4.1]{BM}. Since the statement is slightly different, we give a proof.

\begin{lemma}\label{countK}
Let $N$ be a non-negative integer and $R>0$. Let $U$ be a Borel subset of $\partial X$ such that $\mbox{Diam}(U)\leq e^{-\epsilon NR}.$
There exists a compact $K\subset X$  and a non-negative integer $q$ such that we have $\{g\cdot o|\, g\in W^{\Lambda}_{N,R}(U)\}\subset K$, and thus $|W^{\Lambda}_{N,R}(U)|\leq q.$
\end{lemma}
\begin{proof}
Pick $g,h \in W^{\Lambda}_{N,R}(U)$. We have $(go,ho)_{o}\geq \min \{(go,\xi)_{o},(\xi,\eta)_{o}, (ho,\eta)_{o} \}-3\delta$ with $\xi\in Q_{g}\cap U\neq \varnothing  \neq Q_{h}\cap U \ni \eta $. Thus, Item (3) of Lemma \ref{vitali} implies that for some $r'>0$ we have:
$$ (go,ho)_{o}\geq NR-r'-3\delta.$$ 
Therefore for $g,h\in W^{\Lambda}_{N,R}(U)$ we have
\begin{align*}
d(go,ho)&=d(o,h o)+d(o,g o)-2(go,ho)_{o}\\
&\leq 2(R+r'+3\delta).
\end{align*}
Pick some $h \in W^{\Lambda}_{N,R}(U)$ and define $K:=B_{X}(h \cdot o,R')$ with $R':=4(R+r'+3\delta)$. We have by construction $\{g\cdot o|\, g\in W^{\Lambda}_{N,R}(U)\}\subset K$.
\end{proof}

The following lemma already appear in \cite[Lemma 4.3]{Ga}.  Since our statement is slightly different, we give a proof.

\begin{lemma} \label{measureU} There exists $R>0$ large enough and a constant $C>0$ such that for all non-negative integers $p$ and for all Borel subset $U\subset \partial X$ satisfying $$\mbox{Diam}(U)\leq e^{\epsilon pR}e^{-\epsilon NR},$$ we have $$ |W^{\Lambda}_{N,R}(U)|\leq C e^{ \alpha pR}.$$

\end{lemma}

\begin{remark}
We also have a lower bound
$$|W^{\Lambda}_{N,R}(U)|\geq C^{-1} \nu_{o}(U)e^{\alpha NR},$$ for some $C>0$. 
We give a proof of this fact:
Write $U:= \coprod_{g\in W^{\Lambda }_{N,R}(U)}Q_{g}\cap U$ and by taking the measure, since $Q_{g}\cap U\subset Q_{g}$ we obtain
$$\nu_{o}(U)=\sum_{g\in W^{\Lambda }_{N,R}(U)}\nu_{o}(Q_{g}\cap U)\leq |W^{\Lambda }_{N,R}(U)| C e^{-\alpha NR},$$
where the inequality follows from Item (3) of Lemma \ref{vitali}.
\end{remark}

\begin{proof}
One cannot write $ \coprod_{g\in W^{\Lambda }_{N,R}(U)}Q_{g}\subset U$. We shall thicken $U$ to be able to write the previous inclusion. To do so, we set $r_{N}:=e^{pR}e^{-NR}$ and pick some $\xi \in U$ so that $U\subset B(\xi,r_N)$.\\
Let $g$ be in $W^{\Lambda}_{N,R}(U)$ and so, by definition of $W^{\Lambda}_{N,R}(U)$ we have $Q_{g}\cap U\neq \varnothing$. Since $Q_{g}\subset B(\theta_{o}^{go},5\rho_{N})$ where $\rho_{N}=e^{-\epsilon (NR-r-\delta)}$ is provided by Lemma \ref{vitali} and Lemma \ref{ombre},  we have $B(\theta_{o}^{go},5\rho_{N})\cap U \neq \varnothing$. Pick $R>0$ so that $R>r+\delta+\frac{\log10}{\epsilon}$. We have for  $\eta \in Q_{g}$ 
\begin{align*}
d(\xi,\eta)&\leq d(\xi,\theta_{o}^{go})+d(\theta_{o}^{go},\eta)\\
&\leq r_{N}+2\times 5\rho_{N}\\
&= (e^{\epsilon pR}+10e^{\epsilon (r+\delta)})e^{-\epsilon NR}\\
&\leq 2e^{\epsilon  pR}e^{-\epsilon NR}=2r_{N},
\end{align*}
where the last inequality follows from the choice of $R$.\\
Then for all $g\in S^{\Lambda*}_{N,R}$ we have $Q_{g}\subset B(\xi,2r_{N})$. Combining this fact with the fact that the family $\{Q_{g}\}_{g\in S_{N,R}^\Lambda*}$  is made of disjoint sets, we obtain: $\coprod_{g\in W^{\Lambda }_{N,R}(U)}Q_{g}\subset B(\xi,2r_{N}).$
By taking the measure of the above inclusion we have 
\begin{align*}
\sum_{g\in W^{\Lambda }_{N,R}(U)}\nu_{o}(Q_{g})\leq \nu_{o}(B(\xi,2r_{N})),
\end{align*}
using now the Ahlfors regularity of $\nu_{o}$ and Item (1) of Lemma \ref{vitali} we obtain:
\begin{align*}
\sum_{g\in W^{\Lambda}_{N,R}(U)}C_{\nu}^{-1}e^{-\alpha NR}e^{\alpha(r-\delta)}\leq C_{\nu}2^{\alpha}r_{N}^{\alpha},
\end{align*}
and thus
 \begin{align*}
|W^{\Lambda}_{N,R}(U)|\leq 2^{\alpha}C_{\nu}^{2}e^{\alpha(-r+\delta)}\times e^{ \alpha pR}.
\end{align*}
Set $C:=\max{\{C',2^{\alpha}C_{\nu}^{2}e^{\alpha(-r+\delta)}}\}$ to conclude the proof.

\end{proof}

\subsection{Diameters Lemma}
Let $R>0$.
Let $n$ be a non-negative integer and $k\in \{1,\dots,n\} $, and let $N>n$ be another non-negative integer. Define for $\gamma \in S^{\Gamma}_{n,R}$ and $h\in S^{\Lambda *}_{N,R}$:  
\begin{equation}\label{Uknparts}
U_{k,n}(\gamma,h):= A_{k,R}(\gamma )\cap \gamma Q_{h}.
\end{equation}
Define also 
\begin{equation}
V_{k,n}(\gamma,h):=\gamma^{-1}  A_{k,R}(\gamma )\cap \gamma^{-1}Q_{h}.
\end{equation}

\begin{lemma}\label{diameterlemma}\emph{(Diameters Lemma)}. 
Assume that $R>r+\delta+\frac{\log(5)}{\epsilon}.$
 For all $\gamma \in S^{\Gamma}_{n,R}$, for all $g,h\in S^{\Lambda,*}_{N,R}$ and for any $k\in \{1,\dots,n\} $ we have $\mbox{Diam}(U_{k,n})\leq e^{\epsilon(n-2k)R}e^{-\epsilon NR}$ and $\mbox{Diam}(V_{k,n})\leq e^{\epsilon(2k-n)R}e^{- \epsilon NR}.$ 
 \end{lemma}

\begin{proof}
For $\gamma \in S^{\Gamma}_{n,R}$, for $h\in S^{\Lambda *}_{N,R}$ and for $k\in \{1,\dots,n\}$ write $U:=U_{k,n}(\gamma,h)$.\\
 By the relations of Lemma \ref{lesAi}, we have for $k\in \{1,\dots,n\}$ $$\gamma^{-1}U\subset \big(A_{n-k-1,R}(\gamma^{-1})\cup A_{n-k,R}(\gamma^{-1})\big)\cap Q_{h}.$$ First, we compute the diameter of $\gamma^{-1}U$ with respect to $d_{\gamma^{-1} o}$. Using the conformal relation of the visual metric we have $$d_{\gamma^{-1}o}(\xi,\eta)=d_{o}(\xi,\eta)e^{\frac{1}{2}\epsilon (\beta_{\xi}(o,\gamma^{-1}o)+\beta_{\eta}(o,\gamma^{-1}o) )}.$$ Therefore Lemma \ref{BusemannsurAi} implies that for all $\xi,\eta \in \gamma^{-1} U$ and for all $\gamma \in S^{\Gamma}_{n,R}$ $$d_{\gamma^{-1}o}(\xi,\eta)\leq  5 e^{\epsilon(r+\delta)}e^{-3\epsilon R} e^{-\epsilon NR} e^{(n-2k)\epsilon R},$$
 where the factor $5 e^{\epsilon(r+\delta)}$ comes from Item (3) of Lemma \ref{vitali} and Lemma \ref{ombre}. Then, using the relation $d_{\gamma^{-1} o}(\gamma^{-1} \xi',\gamma^{-1} \eta')=d_{o}( \xi', \eta')$ for $\xi',\eta'\in U$, we deduce that Diam$(U)\leq   e^{\epsilon(n-2k)R} e^{-\epsilon NR} $, by the choice of $R>r+\delta+\frac{\log(5)}{\epsilon}$.  \\

On the other hand, for $\xi,\eta\in  V_{k,n}=\gamma^{-1}( A_{k,R}(\gamma )\cap Q_{g})$ written as $\xi=\gamma^{-1}\xi'$ and $\eta=\gamma^{-1}\eta'$ with $\xi',\eta' \in  A_{k,R}(\gamma )\cap Q_{g}$ we have: 
\begin{align*}
d_{o}(\gamma^{-1}\xi',\gamma^{-1}\eta')&=d_{\gamma o}( \xi', \eta')\\
&=d_{o}(\xi',\eta')e^{\frac{1}{2}\epsilon (\beta_{\xi'}(o,\gamma o)+\beta_{\eta'}(o,\gamma o) )} \\
&\leq 5e^{\epsilon(r+\delta)}e^{-3R\epsilon} e^{-\epsilon NR}  e^{\epsilon(2k+1-n) R}\\
&\leq e^{-\epsilon NR}  e^{\epsilon (2k-n) R},
\end{align*}
 and the proof is done.
 \end{proof}

\subsection{Counting arguments combined with diameters lemma}
Assume that $R>0$ is big enough so that Lemma \ref{measureU}
 and \ref{diameterlemma} hold.\\ Define some subsets of $S^{\Lambda}_{N,R}$ for which we shall estimate the volume growth, for any non-negative integer $N$.\\
For $(\gamma,h)$ fixed in $S^{\Gamma}_{n,R}\times S^{\Lambda}_{N,R}$ define 
\begin{equation}\label{defU}
W^{\Lambda}_{N,k}(\gamma ,h ):=\{ g\in S^{\Lambda*}_{N,R}| U_{k,n}(\gamma,h) \cap Q_{g}\neq \varnothing\}\subset  S^{\Lambda *}_{N,R}.
\end{equation}

For $(\gamma,g)$ fixed in $S^{\Gamma}_{n,R}\times S^{\Lambda}_{N}$

\begin{equation}
\widetilde{W}^{\Lambda}_{N,n-k}(\gamma^{-1} , g ):=\{ h\in S^{\Lambda*}_{N,R}|V_{k,n}(\gamma,g)\cap  Q_{h}\neq \varnothing\}\subset S^{\Lambda *}_{N,R}.
\end{equation}
The above notation depends on $R$. We omit it in the notation to leave it readable.

\begin{remark}\label{inverse}
The notation $\widetilde{W}^{\Lambda}_{N,n-k}(\gamma^{-1} , g )$ is legitimated by Lemma \ref{lesAi} implying that $\widetilde{W}^{\Lambda}_{N,n-k}(\gamma^{-1} , g )\subset W^{\Lambda}_{N,n-k}(\gamma^{-1} , g )\cup W^{\Lambda}_{N,n-k-1}(\gamma^{-1} , g )$.
\end{remark}

We close this section by giving two fundamental estimates based on the previous estimation of the counting argument in Lemma \ref{measureU} and the diameters Lemma.
\begin{prop}\label{countingestimate} There exist $R>0$ and a constant $C>0$ such that for all $N+1>n\geq 1$ we have:

For all $k\in \{1,\dots,n \}$ we have that  $$ |W^{\Lambda}_{N,k}(\gamma ,h )|\leq C e^{\alpha R(n-2k)}, $$ and  $$|\widetilde{W}^{\Lambda}_{N,n-k}(\gamma^{-1} , g ) | \leq  Ce^{\alpha R(2k-n)}.$$
 \end{prop}
 \begin{proof}
 The proof follows from Lemma \ref{diameterlemma} providing an upper bound for the size of diameters of $U_{k,n}(\gamma,h)$ and $V_{k,n}(\gamma,g)$. Therefore, Lemma \ref{measureU} applied to $U_{k,n}(\gamma,h)$ and to $V_{k,n}(\gamma,g)$ concludes the proof after having choosed $
 R$ large enough such that the two cited lemmas hold.
 \end{proof} 
We can easily deduce the following result about finite multiplicity:
\begin{prop}\label{lastcounting}
Let $n\geq 1$ a non-negative integer and let $\gamma \in S^{\Gamma}_{n,R}$. We have the following cases:
\begin{enumerate}
\item There exists a non-negative integer $\frak{n}$ such that for all integers $k$ satisfying $k\in \{1,\dots, \lfloor\frac{n}{2} \rfloor\}$ we have for each $h_{0}\in S^{\Lambda *}_{N,R}$  $$|\{h\in S^{\Lambda *}_{N,R}| W^{\Lambda}_{N,k}(\gamma ,h_{0} )\cap W^{\Lambda}_{N,k}(\gamma ,h )\neq \varnothing\}|\leq \mathfrak{n}.$$
\item  There exists an integer $\frak{n}$ such that for all $k$ satisfying $ k\in \{\lfloor\frac{n}{2} \rfloor,\dots,n \}$ we have for each $g_{0}\in S^{\Lambda *}_{N,R}$ $$|\{g\in S^{\Lambda *}_{N,R}| \widetilde{W}^{\Lambda}_{N,n-k}(\gamma^{-1} ,g_{0} )\cap \widetilde{W}^{\Lambda}_{N,n-k}(\gamma^{-1} ,g )\neq \varnothing\}|\leq \mathfrak{n}.$$
\end{enumerate}
\end{prop}
\begin{proof}
We prove only Item (1), the proof of Item (2) follows from Remark \ref{inverse} .\\
Fix $ h_{0}\in S^{\Lambda *}_{N,R}$. Set $p_{n}:=\lfloor Ce^{n-2k} \rfloor$ where $C$ comes from Item (1) of Proposition \ref{countingestimate}. Then, there exists at most $g_{1},\dots,g_{p_{n}}\in S^{\Lambda*}_{N,R}$ such that 
 $A_{k,R}(\gamma)\cap \gamma Q_{h_{0}}\cap Q_{g_{j}}\neq \varnothing$ for all $j\in \{{1,\dots,p_{n}}\}.$ Since $\lfloor \frac{n}{2} \rfloor \geq k$, Lemma \ref{diameterlemma} implies that $\mbox{Diam}(\gamma^{-1}A_{k,R}(\gamma)\cap  \gamma^{-1}Q_{g})\leq e^{-\epsilon NR}$. 
  By Lemma \ref{countK}, for each $j\in\{1,\dots,p_{n} \}$ there exists a compact $K_{j}$ so that $$\{h\cdot o|\, h \in \widetilde{W}^{\Lambda}_{N,n-k}(\gamma^{-1} , g_{j} ) \}\subset K_{j}.$$  Note $$\{h\in S^{\Lambda}_{N}| W^{\Lambda}_{N,k}(\gamma ,h_{0} )\cap W^{\Lambda}_{N,k}(\gamma ,h )\neq \varnothing\} \subset \cup_{j=1}^{p_{n}} \widetilde{W}^{\Lambda}_{N,n-k}(\gamma^{-1} , g_{j} ).$$ 
  Note that for all $j\in\{1,\dots,p_{n} \}$ the compact sets $K_{j} \ni h_{0}\cdot o$. Define $K:=B_{X}(h_{0}\cdot o, R_{0})$ with $R_{0}=2\mbox{Diam}(K_{j})$. Thus, 
  $$\{h\in S^{\Lambda}_{N}| W^{\Lambda}_{N,k}(\gamma ,h_{0} )\cap W^{\Lambda}_{N,k}(\gamma ,h )\neq \varnothing\} \subset K. $$ Define $\frak{n}:=K\cap \Lambda \cdot o$ to conclude the proof.
\end{proof}

\section{Proof of spectral inequalities}\label{sec8}

Let $(X,d)$ be a roughly-geodesic, $\epsilon$-good, $\delta$-hyperbolic space. Let $\Gamma$ be a discrete group of isometries of $(X,d)$: we shall prove spectral inequalities for $\Gamma$ when this latter acts cocompatly on $(X,d)$. Let $\Lambda$ be another discrete group of isometries of $(X,d)$ acting cocompactly. We assume that $\Gamma$ and $\Lambda$ have the same critical exponents $\alpha$. Let $\nu=(\nu_{x})_{x\in X}$ be the Patterson-Sullivan associated with $\Gamma$. We may take $\Gamma=\Lambda$ but it is more enlightening to think that $\Gamma$ and $\Lambda$ are different. The main interest of proceeding with two different groups is to understand how much the assumption of cocompacity is crucial to prove property RD. \\
Typical examples of such a choice are given by two different cocompact lattices of a rank one semisimple Lie group of non-compact type. \\

 We consider the family of representations $\pi_{s}$ for $s\in [0,1]$ of $\Gamma$ defined in (\ref{repre}). Let $R>0$ such that hold Lemma \ref{vitali}, condition (\ref{condR}), Lemma \ref{measureU} and Lemma \ref{diameterlemma}.
\subsection{Decay of matrix coefficients on the horospherical decomposition}
Let $n \geq 1$ be a non-negative integer.
\subsubsection{General decomposition} We decompose for all $v$ and $w$ in $L^{2}(\partial X,\nu_{o})$ and for all $\gamma \in S^{\Gamma}_{n,R}$ the matrix coefficient $\langle \pi_{s}(\gamma)v,w \rangle$ as follows:
\begin{equation*}
\langle \pi_{s}(\gamma)v,w \rangle=\sum^{n}_{k=1}\int_{A_{k,R}(\gamma)}e^{s\alpha \beta_{\xi}(o,\gamma o)}v(\gamma^{-1} \xi)w(\xi)d\nu_{o}(\xi)=\sum^{n}_{k=1}\langle \pi_{s}(\gamma)v,w \rangle_{|_{A_{k,R}(\gamma)}},
\end{equation*}
where by definition
\begin{equation}\label{coeffsurAk}
\langle \pi_{s}(\gamma)v,w \rangle_{|_{A_{k,R}(\gamma)}}:=\langle \pi_{s}(\gamma)v, \textbf{1}_{A_{k,R}(\gamma)}\cdot w \rangle=\int_{A_{k,R}(\gamma)}e^{s\alpha \beta_{\xi}(o,\gamma o)}v(\gamma^{-1} \xi)w(\xi)d\nu_{o}(\xi).
\end{equation}
Since one of the interest of the partition $(A_{k,R}(\gamma))_{k=1,\dots,n}$ is the control of the Busemann function we obtain, by Lemma \ref{BusemannsurAi} that there exists $C>0$ such that for all $k=1,\dots,n$:

\begin{equation}\label{1eredecompo}
\langle \pi_{s}(\gamma)v,w \rangle_{|_{A_{k,R}(\gamma)}}\leq C \int_{A_{k,R}(\gamma)}e^{s\alpha (2k-n)R}v(\gamma^{-1} \xi)w(\xi)d\nu_{o}(\xi).
\end{equation}

 Let  $v,w$ in the cone of positive vectors  $E^{\Lambda,+}_{N}(\partial X)$ and write $v=\sum_{h\in S^{\Lambda *}_{N,R} } c_{h}\textbf{1}_{h}$ and 
$w=\sum_{g\in S^{\Lambda *}_{N,R} } d_{g}\textbf{1}_{g}$ with $c_{h},d_{g}\in \mathbb{R}^{*}_{+}$.\\
The expression $\langle \pi_{s}(\gamma)v,w \rangle_{|_{A_{k,R}(\gamma)}}$ reads as follows: 
\begin{equation}\label{decomp}
\langle \pi_{s}(\gamma)v,w \rangle_{|_{A_{k,R}(\gamma)}}\leq C e^{s\alpha (2k-n)R}\sum_{(g,h)\in S^{\Lambda*}_{N,R} \times S^{\Lambda*}_{N,R}}d_{g}c_{h} \nu_{o}(A_{k,R}(\gamma ) \cap \gamma Q_{h} \cap Q_{g}).
\end{equation}\\
Define for $N>n$, for $k\in \{ 1,\dots,n\}$ and for $\gamma \in S^{\Gamma}_{n,R}$ the sum \begin{equation}\label{resteducoeff}
\mathcal{R}_{k,N}(\gamma):=\sum_{(g,h)\in S^{\Lambda*}_{N,R} \times S^{\Lambda*}_{N,R}}d_{g}c_{h} \nu_{o}(A_{k,R}(\gamma ) \cap \gamma Q_{h} \cap Q_{g}).
\end{equation}
Therefore we have:
\begin{equation}\label{decompoini}
\langle \pi_{s}(\gamma)v,w \rangle_{|_{A_{k,R}(\gamma)}}\leq C e^{s\alpha (2k-n)R}\mathcal{R}_{k,N}(\gamma).
\end{equation}

 Our goal is to control the size of the support of elements $v,w\in E^{\Lambda}_{N}(\partial X)$ contributing to the sum $\mathcal{R}_{k,N}(\gamma )$. Consider the set 
 \begin{align*}
 I^{\Lambda}_{N}(\gamma)=\{(g,h)\in S^{\Lambda*}_{N,R}\times S^{\Lambda*}_{N,R}| A_{k,R}(\gamma)\cap \gamma Q_{h}\cap Q_{g}\neq \varnothing\}\subset S^{\Lambda*}_{N,R}\times S^{\Lambda*}_{N,R}.
 \end{align*}
Thus, the expression of $\mathcal{R}_{k,N}(\gamma)$ can be written as 
$$\mathcal{R}_{k,N}(\gamma)=\sum_{(g,h)\in I^{\Lambda}_{N}(\gamma)}d_{g}c_{h} \nu_{o}(A_{k,R}(\gamma ) \cap \gamma Q_{h} \cap Q_{g}).$$
 Indeed with notation (\ref{notationSn1}) and (\ref{notationSn2}), since $I^{\Lambda}_{N}(\gamma)\subset S^{\Lambda}_{N,k}(\gamma)  \times \widetilde{S}^{\Lambda}_{N,n-k}(\gamma^{-1}) $ we have 
 \begin{equation}\label{support}
 \mathcal{R}_{k,N}(\gamma)\leq \sum_{ (g,h)\in  S^{\Lambda}_{N,k}(\gamma)  \times \widetilde{S}^{\Lambda}_{N,n-k}(\gamma^{-1}) }d_{g}c_{h} \nu_{o}(A_{k,R}(\gamma ) \cap \gamma Q_{h} \cap Q_{g}).
 \end{equation}

\subsection{Proof of Spectral Inequalities}
The fundamental tool to control the sum is nothing but the Cauchy-Schwarz inequality. 
We start by a proposition based on counting estimates of Section \ref{sec7}.

 \begin{prop}\label{CS1}
There exists $C>0$ such that for all $s\in [0,1]$, for all $\gamma \in S^{\Gamma}_{n,R}$, for all non-negative integers $N>n$, for all $k\in \{1,\dots,n\}$, for all $ v,w\in E^{\Lambda,+}_{N}(\partial X)$, we have 
 \begin{align*}\langle \pi(\gamma)v,w \rangle_{|_{A_{k,R}(\gamma)}}\leq C\times \frac{e^{(\frac{1}{2}-s)\alpha nR}e^{(2s-1)\alpha kR} }{|S^{\Lambda*}_{N,R}|} \bigg(\sum_{(h,g)\in  \widetilde{S}^{\Lambda}_{N,n-k}(\gamma^{-1}) \times  S^{\Lambda}_{N,k}(\gamma) }c^{2}_{h}  d^{2}_{g} \bigg)^{\frac{1}{2}}\end{align*}

 \end{prop}

\begin{proof}

Let $s\in [0,1]$, let $N>n$ be non-negative integers, let $k\in \{1,\dots,n\}$ and $\gamma \in S^{\Gamma}_{n,R}$.
The proof splits in two cases. The first case consists in studying $\mathcal{R}_{k,N}(\gamma)$ with $k\leq \frac{\lfloor n \rfloor}{2}$. According to the notation (\ref{Uknparts}) and (\ref{defU}),  Inequality (\ref{support}) becomes: 
\begin{align*}
\mathcal{R}_{k,N}(\gamma)
&\leq\sum_{h\in \widetilde{S}^{\Lambda}_{N,n-k}(\gamma^{-1})} c_{h} \sum_{g\in W^{\Lambda}_{N}(\gamma,h)}d_{g}\nu_{o}(U_{k,n}(\gamma,h)\cap Q_{g}).
 \end{align*}
 Observe that Item (3) of Lemma \ref{vitali} combined with Ahlfors regularity of $\nu_{o}$ imply the existence of some $C>0$ such that for all $g,h\in S^{\Lambda*}_{N,R}:$ 

\begin{equation}\label{inegshad}
  \nu_{o}(U_{k,n}(\gamma,h)\cap Q_{g})\leq \nu_{o}( Q_{g})\leq  C\times \frac{1}{|S^{\Lambda*}_{N,R}|}.
  \end{equation}
 Therefore we have:
 \begin{align*}
\mathcal{R}_{k,N}(\gamma)&\leq  \frac{C}{|S^{\Lambda*}_{N,R}|} \sum_{h\in \widetilde{S}^{\Lambda}_{N,n-k}(\gamma^{-1})}c_{h} \sum_{g\in W^{\Lambda}_{N}(\gamma,h)}d_{g}.
\end{align*}
 
 Then, performing several times the Cauchy-Schwarz inequality, with an absorbing constant $C$ for the third inequality, we obtain:
 
 \begin{align*}
 \mathcal{R}_{k,N}(\gamma)&\leq  \frac{C}{|S^{\Lambda*}_{N,R}|} \bigg(\sum_{h\in  \widetilde{S}^{\Lambda}_{N,n-k}(\gamma^{-1})}c^{2}_{h} \bigg)^{\frac{1}{2}}\bigg(\sum_{h\in  \widetilde{S}^{\Lambda}_{N,n-k}(\gamma^{-1})} \big(\sum_{g\in W^{\Lambda}_{N}(\gamma,h)}d_{g}\big)^{2} \bigg)^{\frac{1}{2}}\\
&\leq \frac{C }{|S^{\Lambda*}_{N,R}|} \bigg(\sum_{h\in  \widetilde{S}^{\Lambda}_{N,n-k}(\gamma^{-1})}c^{2}_{h} \bigg)^{\frac{1}{2}}  \bigg(\sum_{h\in  \widetilde{S}^{\Lambda}_{N,n-k}(\gamma^{-1})} \big( |W^{\Lambda}_{N}(\gamma,h)|\times \sum_{g\in W^{\Lambda}_{N}(\gamma,h)} d^{2}_{g} \big) \bigg)^{\frac{1}{2}}\\
 \mbox{(by Proposition \ref{countingestimate})}&\leq \frac{C e^{\alpha(\frac{n}{2}-k)R}}{|S^{\Lambda*}_{N,R}|} \bigg(\sum_{h\in  \widetilde{S}^{\Lambda}_{N,n-k}(\gamma^{-1})}c^{2}_{h} \bigg)^{\frac{1}{2}} \bigg(\sum_{h\in  \widetilde{S}^{\Lambda}_{N,n-k}(\gamma^{-1})}\sum_{g\in W^{\Lambda}_{N}(\gamma,h)} d^{2}_{g} \bigg)^{\frac{1}{2}}\\
  \mbox{(by Proposition \ref{lastcounting})}&\leq \frac{C \frak{n} e^{\alpha(\frac{n}{2}-k)R}}{|S^{\Lambda*}_{N,R}|} \bigg(\sum_{(h,g)\in  \widetilde{S}^{\Lambda}_{N,n-k}(\gamma^{-1}) \times  S^{\Lambda}_{N,k}(\gamma) }c^{2}_{h}  d^{2}_{g} \bigg)^{\frac{1}{2}}.
\end{align*}
For $\frac{\lfloor n \rfloor }{2}\leq k$: First observe that using the quasi-invariance of $\nu_{o}$, combined with Lemma \ref{lesAi} and Lemma \ref{BusemannsurAi}, we have for some constant $C>0$:
\begin{align}\label{quasiinv}
\nu_{o}(A_{k,R}(\gamma ) \cap \gamma Q_{h}\cap Q_{g})\leq Ce^{\alpha(n-2k)R}\nu_{o}(\gamma^{-1}A_{k,R}(\gamma ) \cap  Q_{h}\cap \gamma^{-1} Q_{g}).
\end{align}
Therefore, the notation (\ref{Uknparts}), (\ref{defU}), Inequality (\ref{inegshad}) applied with $V_{k,n}(\gamma,g)$ in this case and Inequality (\ref{quasiinv}) just above imply:
\begin{align*}
\mathcal{R}_{k,N}(\gamma )
&\leq Ce^{\alpha(n-2k)R}\sum_{g\in S^{\Lambda}_{N,k}(\gamma)}d_{g} \sum_{h\in \tilde{W}^{\Lambda}_{N}(\gamma^{-1},g) }c_{h}\nu_{o}(V_{k,n}(\gamma,g)\cap Q_{h}) \\
&\leq \frac{Ce^{\alpha(n-2k)R}}{|S^{\Lambda*}_{N,R}|}\sum_{g\in S^{\Lambda}_{N,k}(\gamma)}d_{g} \sum_{h\in \tilde{W}^{\Lambda}_{N}(\gamma^{-1},g) }c_{h}.
\end{align*}
Now, performing again the Cauchy-Scwharz inequality we obtain 

\begin{align*}
R_{k,N}(\gamma )
&\leq\frac{Ce^{\alpha(n-2k)R}}{|S^{\Lambda*}_{N,R}|}\big(\sum_{g\in S^{\Lambda}_{N,k}(\gamma)}d^{2}_{g}\big)^{\frac{1}{2}} \bigg(\sum_{g\in S^{\Lambda}_{N,k}(\gamma)}(\sum_{h\in \tilde{W}^{\Lambda}_{N}(\gamma^{-1},g) }c_{h})^{2}\bigg)^{\frac{1}{2}}\\
&\leq\frac{Ce^{\alpha(n-2k)R}}{|S^{\Lambda*}_{N,R}|} \big(\sum_{g\in S^{\Lambda}_{N,k}(\gamma)}d^{2}_{g}\big)^{\frac{1}{2}} \bigg(\sum_{g\in S^{\Lambda}_{N,k}(\gamma)}\sum_{h\in \tilde{W}^{\Lambda}_{N}(\gamma^{-1},g) }|\tilde{W}^{\Lambda}_{N}(\gamma^{-1},g) |  c^{2}_{h}\bigg)^{\frac{1}{2}}\\
\mbox{(by Proposition \ref{countingestimate})}&\leq\frac{Ce^{\alpha(n-2k)R} e^{\alpha(k-\frac{n}{2})}}{|S^{\Lambda*}_{N,R}|} \big(\sum_{g\in S^{\Lambda}_{N,k}(\gamma)}d_{g}\big)^{\frac{1}{2}} \bigg(\sum_{g\in S^{\Lambda}_{N,k}(\gamma)}\sum_{h\in \tilde{W}^{\Lambda}_{N}(\gamma^{-1},g) }  c^{2}_{h}\bigg)^{\frac{1}{2}}\\
 \mbox{(by Proposition \ref{lastcounting})}&\leq\frac{C\frak{n}e^{\alpha(\frac{n}{2}-k)R} }{|S^{\Lambda*}_{N,R}|} \big(\sum_{g\in S^{\Lambda}_{N,k}(\gamma)}d^{2}_{g}\big)^{\frac{1}{2}} \bigg(\sum_{h\in \tilde{S}^{\Lambda}_{N,n-k}(\gamma^{-1})}c^{2}_{h}\bigg)^{\frac{1}{2}}.
\end{align*}
Thanks the decomposition (\ref{decompoini}) the new expression of a matrix coefficient associated with $v,w\in E^{\Lambda,+}_{N}(\partial X)$ reads as follows: 
\begin{equation}\label{1eredecompo}
\langle \pi_{s}(\gamma)v,w \rangle_{|_{A_{k,R}(\gamma)}}\leq C\frak{n} \frac{e^{(\frac{1}{2}-s)\alpha n R}e^{\alpha (2s-1) kR} }{|S^{\Lambda*}_{N,R}|} \big(\sum_{g\in S^{\Lambda}_{N,k}(\gamma)}d^{2}_{g}\big)^{\frac{1}{2}} \bigg(\sum_{h\in \tilde{S}^{\Lambda}_{N,n-k}(\gamma^{-1})}c^{2}_{h}\bigg)^{\frac{1}{2}},
\end{equation}

and the proof is done.

\end{proof}

\begin{remark}
So far, we have not yet used the assumption of cocompacity of the action $\Gamma$ on $(X,d)$. Thus, Proposition \ref{CS1} holds whenever $\Gamma$ is only a discrete group of isometries of $(X,d)$.
\end{remark}

We can now provide the proof of Theorem \ref{maintheo}:
\begin{proof}
First of all, to prove a spectral inequality dealing with $\pi_{s}(f)$ for $f\in C_{c}(\Gamma)$ we use Lemma \ref{reduction}: we shall control the following expression $\sum_{\gamma\in S^{\Gamma}_{n,R}} f(\gamma) \langle \pi_{s}(\gamma) v,w\rangle$, with $f$ a positive function supported on $S^{\Gamma}_{n,R}$ with some $R>0$ and $v,w\in  E^{\Lambda,+}_{N}(\partial X)$ with $N>n+1$. 

 Proposition \ref{CS1} provides a constant $C>0$ such that for all $k\in \{1,\dots,n \}$ 

\begin{align*}
\sum_{S^{\Gamma}_{n,R}}f(\gamma)\langle \pi_{s}(\gamma)v,w\rangle_{|_{A_{k,R}(\gamma)}}
&\leq   C  \frac{e^{(\frac{1}{2}-s)\alpha nR}e^{(2s-1)\alpha kR}}{|S^{\Lambda *}_{N,R}|}\sum_{\gamma \in S^{\Gamma}_{n,R}} f(\gamma)\bigg(\sum_{(h,g)\in  \widetilde{S}^{\Lambda}_{N,n-k}(\gamma^{-1}) \times  S^{\Lambda}_{N,k}(\gamma) }c^{2}_{h}  d^{2}_{g} \bigg)^{\frac{1}{2}} \\
 \mbox{by the Cauchy-Schwarz inequality}&\leq C{ \frac{ e^{(\frac{1}{2}-s)\alpha nR}e^{(2s-1)\alpha kR} }{|S^{\Lambda*}_{N,R}|}} \|f\|_{2}\bigg(\sum_{\gamma \in S^{\Gamma}_{n,R}} \sum_{(h,g)\in  \widetilde{S}^{\Lambda}_{N,n-k}(\gamma^{-1}) \times  S^{\Lambda}_{N,k}(\gamma) }c^{2}_{h}  d^{2}_{g} \bigg)^{\frac{1}{2}}\\
 \mbox{by Proposition \ref{mfinal}}\mbox{ and (\ref{norm2})}&\leq C\frak{m} e^{(\frac{1}{2}-s)\alpha nR}e^{(2s-1)\alpha kR}  \|f\|_{2}\|v\|_{2}\|w\|_{2}.
\end{align*}
We treat the case  $0<s \neq \frac{1}{2}<1$ first.
Take the sum over $k$ to obtain

\begin{align*}
\sum_{S^{\Gamma}_{n,R}}f(\gamma)\langle \pi_{s}(\gamma)v,w\rangle &\leq C\frak{m}\frac{e^{(\frac{1}{2}-s)\alpha nR}-e^{(s-\frac{1}{2})\alpha nR}}{1-e^{(2s-1)\alpha R}}  \|f\|_{2}\|v\|_{2}\|w\|_{2}.
\end{align*}
 If $s<\frac{1}{2},$ then $1-e^{(2s-1)\alpha R}\geq 1-e^{(2s-1)\alpha }$ since $R>1$. Hence, 
 \begin{align}\label{s1}
 \sum_{S^{\Gamma}_{n,R}}f(\gamma)\langle \pi_{s}(\gamma)v,w\rangle \leq C\frak{m}  \big( 1+\omega_{ \frac{1}{2}-s}(nR)\big) \|f\|_{2}\|v\|_{2}\|w\|_{2}. 
 \end{align}
  If, $s>\frac{1}{2}$ then $1-s<\frac{1}{2}$. Then we have 
 $$\sum_{S^{\Gamma}_{n,R}}f(\gamma)\langle \pi_{s}(\gamma)v,w  \rangle= \sum_{S^{\Gamma}_{n,R}}f(\gamma)\langle \pi_{1-s}(\gamma^{-1})v,w  \rangle=\sum_{S^{\Gamma}_{n,R}} \widecheck{f}(\gamma)\langle \pi_{
 1-s}(\gamma)v,w  \rangle,$$ where $\widecheck{f}(\gamma)=f(\gamma^{-1})$. Since $\|f\|_{2}=\|\widecheck{f}\|_{2}$
 the previous inequality (\ref{s1}) implies for $s>\frac{1}{2}$
 $$\sum_{S^{\Gamma}_{n,R}}f(\gamma)\langle \pi_{s}(\gamma)v,w\rangle \leq C\frak{m}  \big( 1+\omega_{s- \frac{1}{2}}(nR)\big) \|f\|_{2}\|v\|_{2}\|w\|_{2}. $$

For $s=\frac{1}{2}$ we obtain then:
\begin{align*}
\sum_{S^{\Gamma}_{n,R}}f(\gamma)\langle \pi_{s}(\gamma)v,w\rangle &   \leq  C\frak{m}(1+nR)  \|f\|_{2}\|v\|_{2}\|w\|_{2}.
\end{align*}

Hence, we have proved that there exists $R,C>0$ such that for all $n$, for all $f$ supported on $S^{\Gamma}_{n,R}$ and for all $s\in [0,1]$ we have: $\|\pi_{s}(f)\|_{op}\leq C\bigg(1+\omega_{|s-\frac{1}{2}|}(nR )\bigg) \|f\|_{2}.$
\end{proof}

\section{ A remark about $1$-cohomology of slow growth boundary representations }\label{sec9}
Let $\Gamma$ be a group acting cocompactly on $(X,d)$ a roughly geodesic, $\epsilon$-good, $\delta$-hyperbolic space and fix $o$ a base point of $X$. Consider its Patterson-Sullivan measure denoted by $\nu_{o}$ of dimension $\alpha$. Consider the family of representations defined in (\ref{repre}) of Section \ref{sec4}. As an immediate corollary of Theorem \ref{maintheo}, we obtain that there exists $C>0$ such that the representations $\pi_{s}$ satisfies for $ 0 \leq s\neq\frac{1}{2}\leq 1$ and for all $\gamma \in \Gamma$:
\begin{equation}\label{slowgrowth}
 \|\pi_{s}(\gamma)\|_{op}\leq C \frac{e^{\alpha |\frac{1}{2}-s| |\gamma|}}{|1-e^{2s-1}|}.
 \end{equation}

The action of $\Gamma$ on $\partial X \times \partial X$ provides both interesting representations and cocycles. More precisely: consider the measure space $(\partial X \times \partial X,m_{BM})$
 where the $m_{BM}$ is the Bowen-Margulis measure defined as 
 \begin{equation}
 d m_{BM}(\xi,\eta):=\frac{d\nu_{o}(\xi)d\nu_{o}(\eta)}{d^{2\alpha }(\xi,\eta)}.
 \end{equation}
 It tuns out that the diagonal action of $\Gamma \curvearrowright (\partial X \times \partial X,m_{BM}) $ is measure preserving. Thus, we can consider 
 the operator $\pi(\gamma)$ defines as $$\pi(\gamma)F(\xi,\eta)=F(\gamma^{-1}\xi,\gamma^{-1}\eta),$$
 for all $\gamma \in \Gamma$, for $\xi,\eta\in \partial X$ and for $F\in L^{p}(\partial X \times \partial X,m_{BM})$ for $p>0$. Hence, $\pi$ defines an isometric representation on a Banach space:
 \begin{align}\label{pip}
 \pi: \Gamma \rightarrow Iso(L^{p}(\partial X \times \partial X,m_{BM})),
  \end{align}
  for $p>0$.
 In the paper \cite{Ni}, Nica proved the following theorem:
\begin{theo} (B. Nica)
Let $\Gamma$ be a hyperbolic group. Then, for $p$ large enough, the above representation $\pi$ in (\ref{pip}) admits a proper 1-cocycle. 
\end{theo} 
In the same vein, based on the ideas of Nica \cite[Proposition 7.1]{Ni} we give a proof of Theorem \ref{theolast}:
 
\begin{proof}

Consider the representation 

\begin{equation}\rho_{s}:=\pi_{s} \otimes \pi_{s}:\Gamma \rightarrow \mathbb{B}(L^{2}(\partial X \otimes \partial X,\nu_{o}\otimes \nu_{o})),\end{equation} with $\pi_{s}$ defined by the expression (\ref{repre}). Thus, by (\ref{slowgrowth}), the representation $\pi_{s} \otimes \pi_{s}$ is a slow growth representation acting on a Hilbert space satisfying \begin{equation}\label{slowgrowth}
 \|\rho_{s}(\gamma)\|_{op}\leq C \frac{e^{ |1-2s| \alpha|\gamma|}}{|1-e^{2s-1}|^{2}}.
 \end{equation} 

Now, let us define the following $1$-cocycle associated to $\rho_{s}$ for $0\leq s<\frac{\epsilon}{4}$ as:
\begin{equation}
b_{s}:\gamma \in \Gamma \longmapsto \bigg((\xi,\eta)\mapsto \frac{\beta_{\xi}(o,\gamma o)-\beta_{\eta}(o,\gamma o)}{d_{o}^{\frac{2s}{\epsilon}\alpha }(\xi,\eta)}\bigg)\in L^{2}(\partial X \otimes \partial X,\nu_{o}\otimes \nu_{o}).
\end{equation}
The condition $0\leq s<\frac{\epsilon}{4}$ is to ensure that the cocycle in well defined in $L^{2}(\partial X \otimes \partial X,\nu_{o}\otimes \nu_{o})$. See below.\\

Indeed we have:
\begin{align*}
b_{s}(\gamma_{1}\gamma_{2})(\xi,\eta)&= \frac{\beta_{\xi}(o,\gamma_{1} \gamma_{2} o)-\beta_{\eta}(o,\gamma_{1} \gamma_{2} o)}{d_{o}^{\frac{2s}{\epsilon}\alpha }(\xi,\eta)}\\
&= b_{s}(\gamma_{1})(\xi,\eta) +\frac{\beta_{\gamma_{1}^{-1}\xi}( o,\gamma_{2} o)-\beta_{\gamma_{1}^{-1}\eta}( o, \gamma_{2} o)}{d_{o}^{\frac{2s}{\epsilon}\alpha}(\xi,\eta)} \\
&= b_{s}(\gamma_{1})(\xi,\eta) +\frac{e^{s\alpha (\beta_{\xi}(o,\gamma_{1}  o) +\beta_{\xi}(o,\gamma_{1} o)) } \big(\beta_{\gamma_{1}^{-1}\xi}( o,\gamma_{2} o)-\beta_{\gamma_{1}^{-1}\eta}( o, \gamma_{2} o)\big)}{\big(e^{\frac{\epsilon}{2} (\beta_{\xi}(o,\gamma_{1}  o) +\beta_{\xi}(o,\gamma_{1} o))}d_{o}(\xi,\eta)\big)^{\frac{2s}{\epsilon}\alpha }}\\
&= b_{s}(\gamma_{1})(\xi,\eta) +\frac{e^{s\alpha (\beta_{\xi}(o,\gamma_{1}  o) +\beta_{\xi}(o,\gamma_{1} o)) }  \big(\beta_{\gamma_{1}^{-1}\xi}( o,\gamma_{2} o)-\beta_{\gamma_{1}^{-1}\eta}( o, \gamma_{2} o)\big)}{d^{\frac{2s}{\epsilon}\alpha }_{ o}(\gamma^{-1}_{1}\xi,\gamma_{1}^{-1}\eta)}\\
&= b_{s}(\gamma_{1})(\xi,\eta) +\rho_{s}(\gamma_{1})b_{s}(\gamma_{2})(\xi,\eta).
\end{align*}

Observe that 
\begin{align*}
|b_{s}(\gamma)(\xi,\eta)|\leq \frac{2|\gamma|}{d^{\frac{2s}{\epsilon}\alpha }_{o}(\xi,\eta)}.
\end{align*}
By Lemma 5.3 of \cite{Ni}, if $p>\alpha $ we have $d\in L^{p}(\partial X \times \partial X, m_{BM})$. Thus, write 
\begin{align*}
\|b_{s}(\gamma)\|^{2}_{2}&\leq 4 |\gamma|^{2}\int_{\partial X \times \partial X}\frac{1}{d_{o}^{\frac{4s\alpha}{\epsilon} }(\xi,\eta) } d\nu_{o}(\xi)d\nu_{o}(\eta)\\
&= 4|\gamma|^{2}\int_{\partial X \times \partial X}\frac{d_{o}^{p}(\xi,\eta)}{d_{o}^{2\alpha }(\xi,\eta) } d\nu_{o}(\xi)d\nu_{o}(\eta)\\
&=4|\gamma|^{2}\int_{\partial X \times \partial X}d_{o}^{p}(\xi,\eta)dm_{BM}(\xi,\eta),
\end{align*}
with $p=2\alpha(1-\frac{2}{\epsilon}s)$. So we need to have $p>\alpha $, that is to say $s<\frac{\epsilon}{4}$.\\
The fact that $b_{s}$ is proper follows from the computation of Proposition 7.1 of \cite{Ni}. More precisely, Nica's method consists in splitting the integral over the sets $A_{k,R}(\gamma)$. By setting  $s'=\frac{s}{\epsilon}>0$ in the following computation we have
\begin{align*}
\|b_{s}(\gamma)\|^{2}_{2}&=\sum^{n}_{j,k=1}\int_{A_{j,R}(\gamma) \times A_{k,R}(\gamma)}|(\xi,\gamma o)_{o}-(\eta,\gamma o)_{o}|^{2}e^{2s'\alpha (\xi,\eta)}d\nu_{o}(\xi)d\nu_{o}(\eta)\\
&\geq e^{-2s'\alpha \delta}\sum^{n}_{j,k=1}\int_{A_{j,R}(\gamma) \times A_{k,R}(\gamma)}|(\xi,\gamma o)_{o}-(\eta,\gamma o)_{o}|^{2}e^{2s'\alpha \min{\{(\xi,\gamma),(\gamma,\eta)}\} }d\nu_{o}(\xi)d\nu_{o}(\eta)\\
&\geq e^{-2s'\alpha \delta}\sum_{1\leq j<k\leq n}|(k-1-j)R|^{2}e^{(2s'-1)\alpha jR }e^{-k\alpha R}\\
&\geq e^{-2s'\alpha \delta}R^2 e^{-\alpha R}\sum_{1\leq j\leq k\leq n}|(k-j)|^{2}e^{(2s'-1)\alpha jR }e^{-k\alpha R}.
\end{align*}
Set now $S_{n}:=\sum_{1\leq j\leq k\leq n}|(k-j)|^{2}e^{(2s'-1)\alpha jR }e^{-k\alpha R}.$ Observe that
\begin{align*}
 S_{n+1}-S_{n}&\geq \sum^{n-1}_{k=1}|(k-n)R|^{2}e^{(2s'-1)\alpha (n-1)R }e^{-k\alpha R}\\
 & = \sum^{n}_{k=1}|(k-1)R|^{2}e^{(2s'-1)\alpha (n-1)R }e^{-(k+n)\alpha R}\\
 & = e^{-(2s'-1)\alpha R }e^{2s'\alpha nR }\sum^{n-1}_{k=0}|(k-1)R|^{2}e^{-k\alpha  R}\\
 &\geq  R^{2}e^{-(2s'-1)\alpha R } \times (ne^{s'\alpha nR }).
\end{align*}
Hence, $\|b_{s}(\gamma)\|\to+\infty$ exponentially fast for any $0\leq s\neq\frac{1}{2}\leq1$. \\To conclude, the representation  $\pi_{s} \otimes \pi_{s}:\Gamma \rightarrow \mathbb{B}(L^{2}(\partial X \otimes \partial X,\nu_{o}\otimes \nu_{o}))$ admits a proper cocycle if $0\leq s<\frac{\epsilon}{4}$.  \\
In this generality, the previous computations hold for hyperbolic groups possessing property (T). But it is well known that such groups cannot have an unitary representation with a proper 1-cocycle. We refer to \cite{BC} for more details on property (T).
Since $ \rho_{\frac{1}{2}}$ is an unitary representation, there exists $\frac{1}{2}>\varepsilon>0$ such that 
  $\rho_{s}$ for $s\in [0,\frac{1}{2}-\varepsilon[$ admits a proper 1-cocycle.
\end{proof}

\end{document}